\documentclass[english]{amsart}

\usepackage{esint}
\usepackage[svgnames]{xcolor} 
\usepackage[colorlinks,citecolor=red,pagebackref,hypertexnames=false,breaklinks]{hyperref}
\usepackage{pgf,tikz}
\usepackage{pdfsync}

\usepackage{dsfont}
\usepackage{url}
\usepackage[utf8]{inputenc}
\usepackage[T1]{fontenc}
\usepackage{lmodern}
\usepackage{babel}
\usepackage{mathtools}  
\usepackage{amssymb}
\usepackage{lipsum}
\usepackage{mathrsfs}
\usepackage{color}

\newtheorem{theorem}{Theorem}[section]
\newtheorem{proposition}{Proposition}[section]
\newtheorem{lemma}{Lemma}[section]

\newtheorem{corollary}{Corollary}[section]

\numberwithin{equation}{section}

\title[Inverse problems for Schr\"odinger equations]{Inverse problems for Schr\"odinger equations with unbounded potentials}

\thanks{The author is supported by the grant ANR-17-CE40-0029 of the French National Research Agency ANR (project MultiOnde).}

\author[Mourad Choulli]{Mourad Choulli}
\address{Universit\'e de Lorraine, 34 cours L\'eopold, 54052 Nancy cedex, France}
\email{mourad.choulli@univ-lorraine.fr}

\begin{document}

\begin{abstract}
We summarize in these notes the course given at the Summer School of AIP 2019 held in Grenoble from July 1st to July 5th. This course was mainly devoted to the determination of  the unbounded potential in a Schr\"odinger equation from the associated Dirichlet-to-Neumann map (abbreviated to DN map in this text). We establish a stability inequality for potentials belonging to $L^n$, where $n\ge 3$ is the dimension of the space. Next, we prove a uniqueness result  for potentials in $L^{n/2}$, $n\ge 3$, and apply this uniqueness result  to demonstrate a Borg-Levinson type theorem. 

We use a classical approach which is essentially based on the construction of the so-called complex geometric optic solutions (abbreviated to CGO solutions in this text).

\end{abstract}



\maketitle

\tableofcontents

\section{Preliminaries}

\subsection{Notations}

Throughout,  $\Omega$ denotes a bounded Lipschitz domain of $\mathbb{R}^n$, $n\ge 3$, with boundary $\Gamma$. 

The constant $c_\Omega$ will always denote a generic constant only depending on $n$ and $\Omega$ and all the Banach spaces we consider are complex. 

We define
\[
\underline{n}=\frac{2n}{n+2},\quad \overline{n}=\frac{2n}{n-2}
\]
and we observe that
\[
1<\underline{n}<2<\overline{n}.
\]

The duality pairing between a Banach space and its dual  is denoted by the symbol $\langle \cdot ,\cdot \rangle$. Finally, $[\cdot ,\cdot]$ denotes the usual commutator: $[A,B]=AB-BA$.

Let $(X,d\mu)=(\Omega ,dx)$ or $(X,d\mu)=(\Gamma ,dS(x))$, $dx$ is the Lebesgue measure on $\mathbb{R}^n$ and $dS(x)$ is the measure induced by the Lebesgue measure on $\Gamma$. 

For $f,g\in L^2(X)=L^2(X,d\mu )$ and $F,G\in L^2(X)=L^2((X,d\mu) ,\mathbb{C}^n)$, we use the usual notations
\begin{align*}
&(f,g)_{L^2(X)}=\int_X f\overline{g}d\mu,
\\
&(F,G)_{L^2(X)}=\int_X F\cdot \overline{G}d\mu,
\\
&\|f\|_{L^2(X)}=(f,f)_{L^2(X)}^{1/2},
\\
&\|F\|_{L^2(X)}=(F,F)_{L^2(X)}^{1/2}.
\end{align*}

It is well known that, according to Poincar\'e's inequality, $u\rightarrow \|\nabla u\|_{L^2(\Omega )}$ is a norm on $H_0^1(\Omega )$ which is equivalent to the $H^1(\Omega )$-norm. In all of this text, $\|u\|_{H_0^1(\Omega )}$ denotes indifferently one these two equivalent norms.

\subsection{Trace theorem}

We recall that $H^1(\Omega )$ is continuously embedded in $L^{\overline{n}}(\Omega )$:
\[
\|u\|_{L^{\overline{n}}(\Omega )}\le c_\Omega\|u\|_{H^1(\Omega )},\quad u\in H^1(\Omega ).
\]
We denote the norm of this embedding by $\mathfrak{e}_\Omega$.

\begin{lemma}\label{lemma-mult}
Let $V\in L^{n/2}(\Omega )$ and denote by $\mathfrak{b}_V$ the sesquilinear  form defined by
\[
\mathfrak{b}_V(u,v)=\int_\Omega Vu\overline{v}dx,\quad u,v\in H^1(\Omega ).
\]
Then $\mathfrak{b}$ is bounded with
\begin{equation}\label{pr1}
|\mathfrak{b}_V(u,v)|\le\mathfrak{e}_\Omega^2 \|V\|_{L^{n/2}(\Omega )}\|u\|_{H^1(\Omega )}\|v\|_{H^1(\Omega )},\quad u,v\in H^1(\Omega ).
\end{equation}
Moreover, for any $u\in H^1(\Omega)$, $\ell_V(u):v\in H_0^1(\Omega )\mapsto \mathfrak{b}_V(u,v)$ belongs to $H^{-1}(\Omega )$ with
\begin{equation}\label{p2}
\|\ell_V(u)\|_{H^{-1}(\Omega )}\le \mathfrak{e}_\Omega^2 \|V\|_{L^{n/2}(\Omega )}\|u\|_{H^1(\Omega )}.
\end{equation}
\end{lemma}

\begin{proof}
Let $u,v\in H^1(\Omega )$. We get by applying Cauchy-Schwarz's inequality 
\begin{equation}\label{p3}
|\mathfrak{b}_V(u,v)|\le \|Vu^2\|_{L^1(\Omega )}^{1/2}\|Vv^2\|_{L^1(\Omega )}^{1/2}.
\end{equation}
As $n/2$ is the conjugate exponent of $\overline{n}/2$, we obtain from H\"older's inequality, with $w=u$ or $w=v$,
\begin{equation}\label{p4}
\|Vw^2\|_{L^1(\Omega )}^{1/2}\le \|V\|_{L^{n/2}(\Omega)}^{1/2}\|w\|_{L^{\overline{n}}(\Omega )}\le \mathfrak{e}_\Omega \|V\|_{L^{n/2}(\Omega)}^{1/2}\|w\|_{H^1(\Omega )}.
\end{equation}
Then \eqref{p4} with $w=u$ and then with $w=v$ in \eqref{p3} gives \eqref{pr1}.

For $u\in H^1(\Omega )$,  \eqref{pr1} yields
\[
|\ell_V(u)(v)|\le \mathfrak{e}_\Omega^2 \|V\|_{L^{n/2}(\Omega )}\|u\|_{H^1(\Omega )}\|v\|_{H_0^1(\Omega )},\quad v\in H_0^1(\Omega ),
\]
from which we deduce  that $\ell_V(u)\in H^{-1}(\Omega )$ and that \eqref{p2} holds. 
\end{proof}

The following trace theorem is due to Gagliardo (see for instance \cite[Theorem 1.5.1.3, page 38]{Gr}) in which $C^{0,1}(D)$, $D=\overline{\Omega}$ or $D=\Gamma$, denotes the Banach space of Lipschitz continuous functions on $D$.

\begin{theorem}\label{trace-theorem-D}
The mapping $u\in C^{0,1}(\overline{\Omega})\mapsto u_{|\Gamma}\in C^{0,1}(\Gamma)$ has unique bounded extension, denoted by $\gamma_0$, as an operator from $H^1(\Omega )$ onto $H^{1/2}(\Gamma )$. This operator has a bounded right inverse.
\end{theorem}

According to this theorem, for any $f\in H^{1/2}(\Gamma )$, there exists $\mathscr{E}f\in H^1(\Omega )$ so that $\gamma_0\mathscr{E}f=f$ and 
\begin{equation}\label{p5}
\|\mathscr{E}f\|_{H^1(\Omega )}\le c_\Omega \|f\|_{H^{\frac{1}{2}}(\Gamma )}.
\end{equation}

We define on $H^{1/2}(\Gamma )$ the quotient norm
\[
\| f\|_q =\inf \left\{ \|F\|_{H^1(\Omega )};\; \gamma_0F=f \right\}.
\]
In light of \eqref{p5}, this quotient norm is equivalent to the initial norm on $H^{1/2}(\Gamma )$. Henceforward for convenience we will not distinguish these two norms and we use the notation $\|f\|_{H^{1/2}(\Gamma )}$ for both.

We set, for $V\in L^{n/2}(\Omega )$, 
\[
\mathscr{S}_V=\{ u\in H^1(\Omega );\; (-\Delta u+V)u=0\; \mbox{in}\; \Omega\}.
\]

\begin{lemma}\label{lemma-nor}
Let $V\in L^{n/2}(\Omega)$, $u\in \mathscr{S}_V$ and define 
\[
\gamma_1u : f\in H^{\frac{1}{2}}(\Gamma )\mapsto \gamma_1u(f)=\int_\Omega \left(\nabla u\cdot \nabla \overline{F}+Vu\overline{F}\right)dx,
\]
where $F\in H^1(\Omega )$ is arbitrary so that $\gamma_0F=f$. Then $\gamma_1u$ is well defined, belongs to $H^{-1/2}(\Gamma )$ and
\begin{equation}\label{p6}
\|\gamma_1u\|_{H^{-1/2}(\Gamma )}\le c_\Omega \left(1+\|V\|_{L^{n/2}(\Omega)}\right) \|u\|_{H^1(\Omega )}.
\end{equation}
\end{lemma}

\begin{proof} 
Let $F_1,F_2\in H^1(\Omega )$ so that $\gamma_0F_1=\gamma_0F_2=f$. As $F_1-F_2\in H_0^1(\Omega )$, we find a sequence $(\psi _k)\in C_0^\infty (\Omega )$ converging to $F_1-F_2$. Then 
\[
0=\langle (-\Delta +V)u,\psi_k\rangle =\int_\Omega (\nabla u\cdot \nabla \overline{\psi_k}+Vu\overline{\psi_k})dx.
\]
We obtain by passing to the limit, when $k\rightarrow +\infty$,
\[
0=\int_\Omega \left[\nabla u\cdot \nabla (\overline{F}_1-\overline{F}_2)+Vu(\overline{F}_1-\overline{F}_2)\right]dx.
\]
This shows that $\gamma_1u$ is well defined.

In light of  \eqref{pr1}, we get
\[
\left|  \gamma_1u(f)\right| \le c_\Omega \left(1+\|V\|_{L^{n/2}(\Omega)}\right) \|u\|_{H^1(\Omega )}\|F\|_{H^1(\Omega )}.
\]
But $F\in H^1(\Omega )$ is arbitrary so that $\gamma_0F=f$. Whence
\[
\left|  \gamma_1u(f)\right| \le c_\Omega \left(1+\|V\|_{L^{n/2}(\Omega)}\right) \|u\|_{H^1(\Omega )}\|f\|_{H^{\frac{1}{2}}(\Gamma )}
\]
from which we deduce readily that $\gamma_1u\in H^{-1/2}(\Gamma )$ and \eqref{p6} holds.
\end{proof}

We recall that $C^{1,1}(\overline{\Omega})$ denotes the space of functions from $C^1(\overline{\Omega} )$ having their partial derivatives of first order in $C^{0,1}(\overline{\Omega} )$.

When $u\in \mathscr{S}_V\cap C^{1,1}(\overline{\Omega} )$ then one can check, with the help of Green's formula, that $\gamma_1u=\partial_\nu u$, where $\partial_\nu $ denotes the derivative along the unit normal vector field $\nu$ pointing outward $\Omega$. Therefore, $\gamma_1$ can be seen as an extension of $\partial_\nu $ for functions from $\mathscr{S}_V$.

\begin{lemma}\label{form}
Let $F\in H^1(\Omega )$. Then $\Delta F\in \mathscr{D}'(\Omega )$ extends to a bounded linear form on $H_0^1(\Omega )$. 
\end{lemma}

\begin{proof}
We have
\[
\langle \Delta F ,\varphi \rangle=-\int_\Omega \nabla F\cdot \nabla \overline{\varphi} dx,\quad \varphi \in C_0^\infty (\Omega).
\]
Whence
\[
\left| \langle \Delta F ,\varphi \rangle\right|\le \|\nabla F\|_{L^2(\Omega )}\|\varphi \|_{H_0^1(\Omega)},\quad \varphi \in C_0^\infty (\Omega).
\]
The lemma then follows by using the density of $C_0^\infty (\Omega)$ in $H_0^1(\Omega )$.
\end{proof}

This lemma allows us to consider $\Delta F$, $F\in H^1(\Omega )$, as an element of $H^{-1}(\Omega )$.

\subsection{Spectral decomposition}

For $V\in L^{n/2}(\Omega )$ real valued, we define on $H^1(\Omega)\times H^1(\Omega )$ the sesquilinear form $\mathfrak{a}$ by
\begin{align*}
\mathfrak{a}(u,v)&=\int_\Omega \left(\nabla u\cdot \nabla\overline{ v}+Vu\overline{v}\right)dx
\\
&=\int_\Omega \nabla u\cdot \nabla\overline{v}dx+\mathfrak{b}_V(u,v).
\end{align*}

An immediate consequence of Lemma \ref{lemma-mult} is that $\mathfrak{a}$ is bounded with
\[
|\mathfrak{a}(u,v)|\le \left(1+c_\Omega \|V\|_{L^{n/2}(\Omega )}\right)\|u\|_{H^1(\Omega )}\|v\|_{H^1(\Omega )}.
\]
Let $\mathfrak{p}_\Omega$ be the smallest constant so that
\[
\|u\|_{H^1(\Omega )}\le \mathfrak{p}_\Omega \|\nabla u\|_{L^2(\Omega )},\quad u\in H_0^1(\Omega ).
\]

Using the density of $C_0^\infty (\Omega )$ in $L^{n/2}(\Omega )$, we find $\tilde{V}\in C_0^\infty(\Omega )$ so that 
\[
\|V-\tilde{V}\|_{L^{n/2}(\Omega )}\le \frac{1}{2\left[\mathfrak{p}_\Omega \mathfrak{e}_\Omega \right]^2}.
\]
Therefore, for $u\in H_0^1(\Omega)$,
\begin{align*}
\mathfrak{a}(u,u)&\ge \mathfrak{p}_\Omega^{-2} \|u\|_{H^1(\Omega)}^2-\|\tilde{V}\|_{L^\infty (\Omega )}\|u\|_{L^2(\Omega )}^2
\\
&\hskip 3cm -\|V-\tilde{V}\|_{L^{n/2}(\Omega )}\|u\|_{L^{\overline{n}}(\Omega)}^2
\\
&\ge \frac{\mathfrak{p}_\Omega^{-2}}{2}\|u\|_{H^1(\Omega)}^2-\|\tilde{V}\|_{L^\infty (\Omega )}\|u\|_{L^2(\Omega )}^2
\end{align*}

The bounded operator $A_V: H_0^1(\Omega )\rightarrow H^{-1}(\Omega )$ defined by 
\[
\langle A_Vu,v\rangle =\mathfrak{a}(u,v),\quad u,v\in H_0^1(\Omega )
\]
is then self-adjoint (with respect to the pivot space $L^2(\Omega)$) and coercive. Whence,  according to \cite[Theorem 2.37, page 49]{WM}, the spectrum of $A_V$, $\sigma(A_V)$, consists in a sequence 

\[
-\infty< \lambda_V^1\le \lambda_V^2\le \ldots \le \lambda_V^k\le \ldots ,
\]
with
\[
\lambda_V^k\rightarrow \infty \quad  \mbox{as}\; k\rightarrow \infty .
\]

The set $\rho (A_V)=\mathbb{C}\setminus\sigma(A_V)$ is  usually called the resolvent set of the operator $A_V$. We define then
\[
R_V(\lambda )=(A_V-\lambda )^{-1}, \quad \lambda \in \rho (A_V).
\]

Furthermore, $L^2(\Omega )$ admits an orthonormal basis $(\phi_V^k)$ of eigenfunctions, each $\phi_V^k$ is associated to $\lambda_V^k$.

Finally, we note that we have the following Weyl's asymptotic formula (see \cite[Proposition A.1, page 223]{Po})
\begin{equation}\label{weyl}
\lambda_V^k\sim c_\Omega k^{2/n},\quad k\ge 1.
\end{equation}

We denote for convenience  the subset of real valued potentials $V\in L^{n/2}(\Omega )$ so that  $0\in \rho(A_V)$ by $L_\ast^{n/2}(\Omega )$.

\subsection{Non homogenous BVPs and the DN map}

Pick $V\in L_\ast^{n/2}(\Omega )$, $f\in H^{1/2}(\Gamma )$ and let $w =A_V^{-1}(\mathscr{E}f)\in H_0^1(\Omega )$. If $v\in H_0^1(\Omega )$ then we have
\[
-\langle (- \Delta +V)\mathscr{E}f,v\rangle =-\mathfrak{a}(\mathscr{E}f,v)
\]
and hence
\[
\mathfrak{a}(w+\mathscr{E}f,v)=0,\quad v\in H_0^1(\Omega ).
\]
Therefore, one can check that $u=w+\mathscr{E}f\in H^1(\Omega )$ satisfies $-\Delta u+Vu=0$ in $\mathscr{D}'(\Omega )$ and $\gamma_0u=f$. That is $u$ is the solution of the BVP
\begin{equation}\label{bvp}
\left\{
\begin{array}{ll}
-\Delta u+Vu=0\quad \quad \mbox{in}\; \Omega ,\\ \gamma_0u=f.
\end{array}
\right.
\end{equation}

The uniqueness of solutions of the BVP \eqref{bvp} follows from the fact that $A_V$ is invertible.

\begin{theorem}\label{theorembvp}
Let $V\in L_\ast^{n/2}(\Omega )$ and $f\in H^{1/2}(\Gamma )$. Then the BVP \eqref{bvp} has a unique solution $u_V(f)\in H^1(\Omega )$ with
\begin{equation}\label{ine4}
\|u_V(f)\|_{H^1(\Omega )}\le c_\Omega \left(1+\|V\|_{L^{n/2}(\Omega )}\right)\|f\|_{H^{1/2}(\Gamma )}.
\end{equation}
\end{theorem}

Let $V\in L_\ast^{n/2}(\Omega )$. From Theorem \ref{theorembvp}, for $f\in H^{\frac{1}{2}}(\Gamma )$, we have $u_V(f)\in \mathscr{S}_V$. Therefore, in light of Lemma \ref{lemma-nor},  we can define the DN map $\Lambda_V$ by
\[
\Lambda_V:f\in H^{1/2}(\Gamma )\mapsto \gamma_1u_V(f)\in H^{-1/2}(\Gamma).
\]

From inequalities \eqref{p6}  and \eqref{ine4} we have $\Lambda_V\in \mathscr{B}\left(H^{1/2}(\Gamma ),H^{-1/2}(\Gamma )\right)$ and
\begin{equation}\label{ine7}
\|\Lambda_V\|_{\mathscr{B}\left(H^{1/2}(\Gamma ),H^{-1/2}(\Gamma )\right)}\le c_\Omega \left(1+\|V\|_{L^{n/2}(\Omega )}\right).
\end{equation}

The following integral identity will be useful in the sequel.

\begin{lemma}\label{integral-identity}
Let $V,\tilde{V}\in L_\ast^{n/2}(\Omega )$, $u\in \mathscr{S}_V$ and $\tilde{u}\in \mathscr{S}_{\tilde{V}}$. Then we have
\begin{equation}\label{ii}
\int_\Omega ({\tilde{V}}-V)u\overline{\tilde{u}}dx =\langle \left(\Lambda_{\tilde{V}}-\Lambda_V\right)(\gamma_0u),\gamma_0\tilde{u}\rangle .
\end{equation}
\end{lemma}

\begin{proof}
Let $v=u_{\tilde{V}}(\gamma_0u)$. From the definition of $\gamma_1$ in Lemma \ref{lemma-nor}, we have
\begin{equation}\label{ii1}
\langle \gamma_1(u-v),\gamma_0 \tilde{u}\rangle = \int_\Omega (Vu-\tilde{V}v)\overline{\tilde{u}}dx+\int_\Omega \nabla (u-v)\cdot \nabla \overline{\tilde{u}}dx 
\end{equation}
and
\begin{equation}\label{ii2}
0=\langle \gamma_1{\tilde{u}},\gamma_0(u-v)\rangle = \int_\Omega \tilde{V} \tilde{u}(\overline{u}-\overline{v})dx+\int_\Omega \nabla \tilde{u}\cdot \nabla (\overline{u}-\overline{v})dx.
\end{equation}

Identity \eqref{ii2} yields
\[
\int_\Omega \nabla \tilde{u}\cdot \nabla (\overline{u}-\overline{v})dx=-\int_\Omega \tilde{V} \tilde{u}(\overline{u}-\overline{v})dx.
\]
Equivalently, we have 
\[
\int_\Omega \nabla \overline{\tilde{u}}\cdot \nabla (u-v)dx=-\int_\Omega \tilde{V} \overline{\tilde{u}}(u-v)dx.
\]
This inequality in \eqref{ii1} gives
\[
\langle \gamma_1(u-v),\gamma_0 \tilde{u}\rangle = \int_\Omega (V-{\tilde{V}})u\overline{\tilde{u}}dx.
\]
We end up getting the expected identity by noting that 
\[ 
\langle \gamma_1(v-u),\gamma_0 \tilde{u}\rangle=\langle \left(\Lambda_{\tilde{V}}-\Lambda_V\right)(\gamma_0u),\gamma_0\tilde{u}\rangle.
\]
\end{proof}

Even if it is not always necessary, we assume in the rest of this text that $\Omega$ is of class $C^{1,1}$.

\subsection{DN map for transposition solutions}\label{sub-ts}

We recall that the space $H_\Delta (\Omega)$ is defined by
\[
H_\Delta (\Omega )=\{ u\in L^2(\Omega); \; \Delta u\in L^2(\Omega )\}.
\]

This space, equipped with its natural norm
\[
\|u\|_{H_\Delta (\Omega)}=\|u\|_{L^2 (\Omega)}+\|\Delta u\|_{L^2(\Omega)},
\]
is a Banach space.

For this space we have the following trace theorem

\begin{lemma}\label{lemma-ts1}
$($\cite{BU}$)$ The mapping $u\in C_0^\infty (\overline{\Omega})\mapsto (u_{|\Gamma},\partial_\nu u_{|\Gamma})$ extends to a bounded operator  
\[
u\in H_\Delta (\Omega )\mapsto (\gamma_0u,\gamma_1u)\in H^{-1/2}(\Gamma )\times H^{-3/2}(\Gamma ).
\]
Moreover, where $V\in L^\infty(\Omega ,\mathbb{R})$, the following generalized Green's formula holds
\begin{equation}\label{ts1}
\int_\Omega (-\Delta +V)u\overline{v}=\int_\Omega u(-\Delta +V)\overline{v}+\langle \gamma_0u,\gamma_1v\rangle-\langle \gamma_1u,\gamma_0v\rangle,
\end{equation}
for $u\in H_\Delta (\Omega)$ and $v\in H^2(\Omega )$.
\end{lemma}

The existence of transposition solutions is guaranteed by the following theorem.

\begin{theorem}\label{theorem-ts1}
$($\cite{CKS}$)$ Fix $M>0$. Then there exists a constant $C>0$, only depending on $\Omega$ and $M$ so that, for any $V\in L^\infty (\Omega )\cap L_\ast^{n/2}(\Omega )$ satisfying $\|V\|_{L^\infty(\Omega)}\le M$ and $f\in H^{-1/2}(\Gamma )$, we find  a unique $u_V(f)\in H_\Delta (\Omega)$ solution of the non homogenous BVP
\[
(-\Delta +V)u=0\; \mbox{in}\; \Omega \quad \mbox{and}\quad \gamma_0u=f.
\]
Furthermore,
\begin{equation}\label{ts2}
\|u_V(f)\|_{H_\Delta (\Omega)}\le C\|f\|_{H^{-1/2}(\Gamma )}.
\end{equation}
\end{theorem}

In light of the trace theorem in Lemma \ref{lemma-ts1}, we have $\gamma_1u_V(f)\in H^{-3/2}(\Gamma)$ and
\[
\|\gamma_1u_V(f)\|_{H^{-3/2}(\Gamma)}\le C\|f\|_{H^{-1/2}(\Gamma )}.
\]

Therefore, the operator
\[
\Lambda_V: f\in H^{-1/2}(\Gamma )\mapsto \gamma_1u_V(f)\in H^{-3/2}(\Gamma)
\]
is bounded with
\[
\|\Lambda _V\|_{\mathscr{B}\left(H^{-\frac{1}{2}}(\Gamma ),H^{-\frac{3}{2}}(\Gamma )  \right)}\le C.
\]

It is worth observing that $\Lambda_V-\Lambda_{\tilde{V}}$ is a smoothing operator. Indeed, for $f\in H^{-1/2}(\Gamma )$, $u=u_V(f)-u_{\tilde{V}}(f)$ solves the BVP
\[
(-\Delta +V)u=(\tilde{V}-V)u_{\tilde{V}}(f)\; \mbox{in}\; \Omega \quad \mbox{and}\quad \gamma_0u=0.
\]

The usual $H^2$ regularity yields $u\in H^2(\Omega )$. Furthermore, we have the estimate
\[
\|u\|_{H^2(\Omega )}\le C_0\|u_{\tilde{V}}(f)\|_{L^2(\Omega )}\le C_1\|f\|_{H^{-1/2}(\Gamma )},
\]
the constants $C_0$ and $C_1$ depend only on $\Omega$, $M$ and $V$.

This estimate shows that $\Lambda_V-\Lambda_{\tilde{V}}$ defines a bounded operator from $H^{-1/2}(\Gamma )$ into $H^{1/2}(\Gamma )$.

\section{Uniqueness result for bounded potentials}

In this section we discuss the uniqueness result  in the case of bounded potentials. The objective is on one hand to understand the main steps to establish this uniqueness result using CGO solutions. On the other hand, the analysis we used for bounded potentials case will serve to explain what modifications are necessary to tackle the case of  unbounded potentials.

Fix $\xi \in \mathbb{S}^n$, $V\in L^\infty(\Omega )$ and, for $h>0$, consider the operator
\[
P_h=P_h(V,\xi )=e^{x\cdot \xi /h}h^2(-\Delta +V)e^{-x\cdot \xi/h}.
\]
Elementary computations show that $P_h$ has the form
\[
P_h=-h^2\Delta +2h\xi \cdot \nabla -1+h^2V.
\]

The following Carleman inequality will be used to construct CGO solutions.

\begin{lemma}\label{lemma-bp1}
Let $M>0$. Then there exist two constants $C>0$ and $h_0>0$, only depending on $M$ and $\Omega$, so that, for any $V\in L^\infty (\Omega )$ satisfying $\|V\|_{L^\infty (\Omega)}\le M$, $0<h<h_0$ and $u\in C_0^\infty (\Omega )$, we have
\begin{equation}\label{bp1}
h\|u\|_{L^2(\Omega )}\le C\|P_hu\|_{L^2(\Omega )}.
\end{equation}
\end{lemma}

\begin{proof}
Denote by $P_h^0$ the operator $P_h$ when $V=0$. In that case, for $u\in C_0^\infty (\Omega )$, we have
\begin{align}
 \|P_h^0u\|_{L^2(\Omega )}^2&=\|(h^2\Delta +1)u\|_{L^2(\Omega)}^2\label{bp2}
 \\
 &-4h\Re ((h^2\Delta +1)u,\xi\cdot \nabla u)_{L^2(\Omega)}+h^2\|\xi \cdot \nabla u\|_{L^2(\Omega)}^2.\nonumber
 \end{align}
Making integrations by parts, we can show that
\[
\Re ((h^2\Delta +1)u,\xi\cdot \nabla u)_{L^2(\Omega)}=0
\]

This in \eqref{bp2} yields
\begin{equation}\label{bp3}
 \|P_h^0u\|_{L^2(\Omega )}^2\ge \|\xi \cdot \nabla u\|_{L^2(\Omega)}^2.
 \end{equation}

On the other hand, from the proof of the usual Poincar\'e's inequality, we have
\[
\|\xi \cdot \nabla u\|_{L^2(\Omega)}^2\ge c_\Omega \|u\|_{L^2(\Omega )}.
\]
This and \eqref{bp3} give \eqref{bp1} for $P_h^0$.

Now, for $V\in L^\infty (\Omega )$ satisfying $\|V\|_{L^\infty (\Omega)}\le M$, we have
\[
\|P_h^0\|_{L^2(\Omega )}\le \|P_hu\|_{L^2(\Omega)} +h^2M\|u\|_{L^2(\Omega)}
\]
and then 
\[
Ch\|u\|_{L^2(\Omega )}\le \|P_hu\|_{L^2(\Omega )}+h^2M\|u\|_{L^2(\Omega)}.
\]

The expected inequality then follows by observing that the second term in the right hand side can be absorbed by the left hand side, provided that $h$ is sufficiently small.
\end{proof}

\begin{proposition}\label{proposition-bp1}
Let $M>0$. Then there exist two constants $C>0$ and $h_0>0$, only depending on $M$ and $\Omega$, so that, for any $V\in L^\infty (\Omega )$ with $\|V\|_{L^\infty (\Omega)}\le M$ and $0<h<h_0$, we find $w\in L^2(\Omega )$ satisfying $\left[e^{x\cdot \xi /h}(-\Delta +V)e^{-x\cdot \xi/h}\right]w=f$ and
\begin{equation}\label{bp6}
\|w\|_{L^2(\Omega )}\le Ch\|f\|_{L^2(\Omega )}.
\end{equation}
\end{proposition}

\begin{proof}
Let $V\in L^\infty (\Omega )$ satisfying $\|V\|_{L^\infty (\Omega)}\le M$ and $\xi \in \mathbb{S}^{n-1}$. Let $H=P_h^\ast (C_0^\infty (\Omega ))$ that we consider as a subspace of $L^2(\Omega )$. Noting that if $P_h=P_h (V,\xi)$ then $P_h^\ast=P_h(\overline{V},-\xi)$. Therefore inequality \eqref{bp1} is satisfied if $P_h$ is substituted by $P_h^\ast$.

Pick $f\in L^2(\Omega )$ and define on $H$ the linear form
\[
\ell (P_h^\ast v)=(v,h^2f)_{L^2(\Omega)},\quad v\in C_0^\infty (\Omega ).
\]

According to Lemma \ref{lemma-bp1}, $\ell$ is well defined and bounded with 
\[
|\ell (P_h^\ast v)|\le h^2\|f\|_{L^2(\Omega )}\|v\|_{L^2(\Omega )}\le Ch\|f\|_{L^2(\Omega )}\|P_h^\ast v\|_{L^2(\Omega )}.
\]
Whence, according to Hahn-Banach extension theorem, there exist a linear form $L$ extending $\ell$ to $L^2(\Omega )$ so that $\|L\|_{\left[L^2(\Omega )\right]'}=\|\ell\|_{H}$. In consequence,
\begin{equation}\label{bp4}
\|L\|_{\left[L^2(\Omega )\right]'}\le Ch\|f\|_{L^2(\Omega )}.
\end{equation}

In light of Riesz's representation theorem, we find $w\in L^2(\Omega )$ so that 
\begin{equation}\label{bp5}
\|w\|_{L^2(\Omega )}=\|L\|_{\left[L^2(\Omega )\right]'}
\end{equation}
and
\[
(P_h^\ast v,w)_{L^2(\Omega )}=L(P_h^\ast v)=\ell (P_h^\ast v)=(v,h^2f)_{L^2(\Omega)},\quad v\in C_0^\infty (\Omega ).
\]
Hence $\left[e^{x\cdot \xi /h}(-\Delta +V)e^{-x\cdot \xi/h}\right]w=f$. Finally, we note that \eqref{bp6} is obtained by combining \eqref{bp4} and \eqref{bp5}. \end{proof}

\begin{theorem}\label{theorem-bp1}
Let $M>0$. Then there exist $C>0$ and $h_0>0$, only depending on $n$, $\Omega$ and $M$, so that, for any $V\in L^\infty (\Omega )$ with $\|V\|_{L^\infty (\Omega )}\le M$, $\xi,\zeta \in \mathbb{S}^{n-1}$ satisfying $\xi\bot\zeta$  and $0<h\le h_0$, the equation
\[
(-\Delta +V)u=0\quad \mbox{in}\; \Omega 
\]
admits a solution $u\in H_\Delta (\Omega )$ of the form
\[
u=e^{-x\cdot\left(\xi +i\zeta\right)/h}(1+v),
\]
where $v\in L^2(\Omega )$ satisfies
\[
\|v\|_{L^2(\Omega )}\le Ch.
\]
\end{theorem}

\begin{proof}
Simple computations show that $v$ should verify
\begin{align*}
&\left[e^{x\cdot \xi /h}(-\Delta +V)e^{-x\cdot \xi/h}\right]\left(e^{-ix\cdot \zeta /h}v\right)
\\
&\qquad =-\left[e^{x\cdot \xi /h}(-\Delta +V)e^{-x\cdot \xi/h}\right]\left(e^{-ix\cdot \zeta /h}\right)=-Ve^{-ix\cdot \zeta/h}.
\end{align*}
By Proposition \ref{proposition-bp1}, we find $w\in L^2(\Omega )$ so that  
\[
\left[e^{x\cdot \xi/h }(-\Delta +V)e^{-x\cdot \xi/h}\right]w=-Ve^{-ix\cdot \zeta/h}.
\]
 Then $v=e^{ix\cdot \zeta /h}w$ possesses the required properties.
\end{proof}

We introduce the notation
\[
\mathfrak{S}_V=\{ u\in L^2(\Omega );\; (-\Delta+V)u=0\}.
\]
and we observe that $\mathfrak{S}_V\subset H_\Delta (\Omega )$.

Pick $V,\tilde{V}\in L^\infty_\ast(\Omega )=L^\infty(\Omega )\cap L_\ast^{n/2}(\Omega )$.

Let $\tilde{u}\in \mathfrak{S}_{\tilde{V}}$, $\hat{u}=u_V(\gamma_0\tilde{u})$ and $w=\hat{u}-\tilde{u}$. Taking into account that $\tilde{u}=u_{\tilde{V}}$, we obtain as in Subsection \ref{sub-ts} that $w\in H^2(\Omega )$. We then apply the generalized Green's formula in Lemma \ref{lemma-ts1}, to $u\in \mathfrak{S}_V$ and $v=w$. We get
\begin{equation}\label{bp6}
\int_\Omega (\tilde{V}-V)u\tilde{u}dx= \left\langle \gamma_0u, (\Lambda_V-\Lambda_{\tilde{V}})(\gamma_0\overline{\tilde{u}})\right\rangle ,
\end{equation}
where we used $\gamma_1w=(\Lambda_V-\Lambda_{\tilde{V}})(\gamma_0\tilde{u})$.

We now prove the following uniqueness result.

\begin{theorem}\label{theorem-bp1}
Let $V,\tilde{V}\in L^\infty_\ast(\Omega )$ so that $\Lambda_V=\Lambda_{\tilde{V}}$. Then $V=\tilde{V}$.
\end{theorem}

\begin{proof}
We get from \eqref{bp6}
\begin{equation}\label{bp7}
\int_\Omega (\tilde{V}-V)u\tilde{u}dx= 0,\quad u\in \mathfrak{S}_V,\; \tilde{u}\in \mathfrak{S}_{\tilde{V}}.
\end{equation}

Let $k,\tilde{k}\in \mathbb{R}^n\setminus\{0\}$ and $\xi \in \mathbb{S}^{n-1}$ so that $k\bot \tilde{k}$, $k\bot \xi$ and $\tilde{k}\bot \xi$. We assume that $|\tilde{k}|=\rho$ is sufficiently large in such a way that
\[
\frac{1}{(|k|^2/4+\rho^2)^{1/2}}=h=h(\rho)\le h_0,
\]
where $h_0$ is as in the preceding theorem. Let then
\[
\zeta =h(k/2+\tilde{k}),\quad \tilde{\zeta}=h(k/2-\tilde{k}).
\]
Clearly, $\zeta,\tilde{\zeta}\in \mathbb{S}^{n-1}$, $\zeta \bot \xi$, $\tilde{\zeta} \bot \xi$ and $\zeta+\tilde{\zeta}=hk$.

According to Theorem \ref{theorem-bp1}, we can take $u\in \mathfrak{S}_V$ in \eqref{bp7} of the form
\[
u=e^{-x\cdot\left(\xi +i\zeta\right)/h }(1+v),
\]
where $v\in L^2(\Omega )$ satisfies
\begin{equation}\label{bp7.1}
\|v\|_{L^2(\Omega )}\le Ch.
\end{equation}
Similarly, we can choose $\tilde{u}\in \mathfrak{S}_{\tilde{V}}$ in \eqref{bp7} of the form
\[
\tilde{u}=e^{-x\cdot\left(- \xi +i\tilde{\zeta}\right)/h }(1+\tilde{v}),
\]
where $\tilde{v}\in L^2(\Omega )$ satisfies
\begin{equation}\label{bp7.2}
\|\tilde{v}\|_{L^2(\Omega )}\le Ch.
\end{equation}

This particular choice of $u$ and $\tilde{u}$ in \eqref{bp7} gives
\[
\int_\Omega (\tilde{V}-V)e^{-ix\cdot k }dx=-\int_\Omega e^{-ix\cdot k}(v+\tilde{v}+v\tilde{v})dx.
\]
Inequalities \eqref{bp7.1} and \eqref{bp7.2} then yield
\begin{equation}\label{bp8}
\left| \int_\Omega (\tilde{V}-V)e^{-ix\cdot k }dx\right|\le Ch(\rho),
\end{equation}
where the constant $C$ is independent of $\rho$.

Passing to the limit in \eqref{bp8}, when $\rho$ goes to $\infty$, we find
\[
\mathscr{F}((\tilde{V}-V)\chi_\Omega) (k )=0,\quad  k \in \mathbb{R}^n\setminus\{0\}
\]
and hence $V=\tilde{V}$. 
\end{proof}

If we want to deal with the case of unbounded potentials then the main difficulty lies on the fact that $\mathfrak{S}_V$ is no longer a subspace of $H_\Delta (\Omega )$. To overcome this difficulty, we need to construct $H^1$-CGO solutions (think to the Sobolev embedding $H^1(\Omega )\hookrightarrow L^{\overline{n}}(\Omega )$). To this end, due to a duality argument, we need a Carleman inequality involving an $H^1$-norm of $u$ and the $L^2$-norm of $P_hu$ together a Carleman inequality involving the $L^2$-norm of $u$ and a $H^{-1}$-norm of $P_hu$. This can be easily seen by checking the proof of Proposition \ref{proposition-bp1}. We are going to establish such Carleman inequalities  in the coming section.

\section{ Stability estimate for $L^n$ potentials}

This section consists in an adaptation of the results in \cite{DKSU}. Some technical results are left without proof. We refer to \cite{DKSU} and reference therein for detailed proof.

\subsection{Carleman inequalities}

Let $\mathcal{O}$ be an arbitrary bounded open subset of $\mathbb{R}^n$ so that $\Omega \Subset \mathcal{O}$. 

$\varphi \in C^\infty (\mathcal{O},\mathbb{R})$ is called a limiting Carleman weight for the  Laplace operator if $\nabla \varphi \not= 0$ in $\mathcal{O}$ and the Poisson bracket condition
\[
\{\overline{p_\varphi} ,p_\varphi\}=0\quad \mbox{when}\; p_\varphi =0
\]
holds, where $p_\varphi$ is the principal symbol in semiclassical Weyl quantization of the operator $e^{\varphi/h}(-h^2\Delta )e^{-\varphi/h}$. We note that $p_\alpha$ is obtained easily by substituting in 
\[
e^{\varphi/h}(-h^2\Delta )e^{-\varphi/h}=-h^2\Delta +2h\nabla \varphi \cdot \nabla -|\nabla \varphi |^2+h\Delta \varphi
\]
$-ih\partial_j$ by $\xi_j$. We find in a straightforward manner that
\[
p_\alpha=p_\alpha (x,\xi) =|\xi|^2 +2i\nabla \varphi \cdot \xi -|\nabla \varphi |^2 .
\]

It is worth mentioning that $\varphi (x)=\xi \cdot x$, $x\in \mathbb{R}^n$, with $\xi \in \mathbb{S}^{n-1}$ and $\varphi (x)=\log |x-x_0|$, $x_0\not\in \mathcal{O}$ are two examples of limiting Carleman weights for the Laplace operator. It is also important to observe that if $\varphi$ is a limiting Carleman weight for the  Laplace operator then so is $-\varphi$.

Hereafter, $H_{scl}^1(\mathcal{O} )$ denotes the usual space $H^1(\mathcal{O} )$ when it is endowed with the semi-classical norm
\[
\|u\|_{H^1_{scl}  (\mathcal{O})}=\left(\| u\|_{L^2(\mathcal{O})}^2+\|h\nabla u\|_{L^2(\mathcal{O})}\right)^{1/2}.
\]

For $0<\epsilon <1$, set $\varphi_\epsilon = \varphi+(h/(2\epsilon))\varphi ^2$.

\begin{lemma}\label{lemma-ci}
There exist three constants $C>0$, $0<\epsilon_0<1$ and $0<\delta <1$, only depending on $n$, $\Omega$ and $\varphi$,  so that, for any $0<\epsilon \le \epsilon_0$, $0<h/\epsilon \le\delta$ and $u\in C_0^\infty (\Omega)$, we have
\begin{equation}\label{cin1}
h \left\|e^{\varphi_\epsilon/h}u\right\|_{H^1_{scl} (\Omega )}\le C\sqrt{\epsilon}\left\|e^{\varphi_\epsilon /h}h^2\Delta u\right\|_{L^2(\Omega)}.
\end{equation}
\end{lemma}

If $V\in L^n(\Omega)$ and $u\in H^1(\Omega )$, then one can check that 
\[
h\|e^{\varphi_\epsilon/h}Vu\|_{L^2(\Omega )}\le C\|V\|_{L^n (\Omega )}\|e^{\varphi_\epsilon/h}u\|_{H^1_{scl} (\Omega)}.
\]
the constant $C$ only depends on $n$, $\Omega$ and $\varphi$. Therefore we get from \eqref{cin1}
\[
Ch \left\|e^{\varphi_\epsilon/h}u\right\|_{H^1_{scl} (\Omega )}\le \sqrt{\epsilon}\left\|e^{\varphi_\epsilon/h}h^2(-\Delta +V)u\right\|_{L^2(\Omega)}+\sqrt{\epsilon}h\left\|e^{\varphi_\epsilon/h}u\right\|_{H^1_{scl} (\Omega)}.
\]
In this inequality, if $\epsilon$ is sufficiently small, we can absorb the second term in the right hand side by left hand side. We then obtain

\begin{corollary}\label{corollary-ci}
Let $M>0$. There exist a constant $C>0$ and $h_0>0$, only depending on $n$, $\Omega$, $\varphi $ and $M$, and $0<\delta <1$, only depending on $n$, $\Omega$ and $\varphi$ so that, for any $V\in L^n (\Omega )$ with $\|V\|_{L^n(\Omega )}\le M$, $0<\epsilon \le \epsilon_0$, $0<h/\epsilon \le\delta$ and  $u\in C_0^\infty (\Omega)$, we have
\begin{equation}\label{ci1}
h \|e^{\varphi_\epsilon/h}u\|_{H^1_{scl} (\Omega )}\le C\sqrt{\epsilon} \|e^{\varphi_\epsilon/h}h^2(-\Delta u+V)u\|_{L^2(\Omega)}.
\end{equation}
\end{corollary}

For $s\in \mathbb{R}$ define the semiclassical Bessel potential on $\mathbb{R}^n$ by
\[
J^s=(1-h^2\Delta )^{s/2}.
\]
It is well know that $J^s$ commute with $-\Delta$ and $J^{s+t}=J^sJ^t$, $s,t\in \mathbb{R}$.

We recall that the closure of $C_0^\infty (\mathbb{R}^n)$ with respect to the norm
\[
\|u\|_{H_{scl}^s(\mathbb{R}^n)}=\|J^su\|_{L^2(\mathbb{R}^n)}
\]
is usually denoted by $H_{scl}^s(\mathbb{R}^n)$, and that the dual of $H_{scl}^s(\mathbb{R}^n)$ is isometricaly isomorphic to $H_{scl}^{-s}(\mathbb{R}^n)$.

Viewing $J^s$ as a semiclassical pseudodifferential operator of order $s$, one can get the following pseudolocal estimate: if $\psi ,\chi \in C_0^\infty (\mathbb{R}^n)$ with $\chi =1$ near $\mbox{supp}(\psi)$ and if $s,\alpha ,\beta\in \mathbb{R}$ and $\ell\in \mathbb{N}$, then
\begin{equation}\label{ple}
\| (1-\chi)J^s(\psi u)\|_{H^\alpha_{scl}(\mathbb{R}^n)}\le C_\ell h^\ell\|u\|_{H_{scl}^\beta(\mathbb{R}^n)}.
\end{equation}

Also, if $P$ is a first order semiclassical operator in $\mathbb{R}^n$, then the commutator estimate holds
\begin{equation}\label{ce}
\|[A,J^s]u\|_{L^2(\mathbb{R}^n)}\le Ch\|u\|_{H^s(\mathbb{R}^n)}.
\end{equation}

Let us consider the semiclassical second order operator
\[
P_{\varphi_\epsilon} =e^{\varphi_\epsilon/h}h^2(-\Delta +V)e^{-\varphi_\epsilon/h}.
\]
That is
\[
P_{\varphi_\epsilon} =-h^2\Delta +2h\nabla \varphi_\epsilon  \cdot \nabla -|\nabla \varphi_\epsilon |^2+h\Delta \varphi_\epsilon +h^2V.
\] 

Then inequality \eqref{ci1} may be rewritten in the following form, with $0<\epsilon \le \epsilon_0$, $0<h/\epsilon \le\delta$ and $u\in C_0^\infty (\Omega )$,
\begin{equation}\label{ci2}
h\|u\|_{H^1_{scl} (\Omega )}\le C\sqrt{\epsilon} \|P_{\varphi_\epsilon} u\|_{L^2(\Omega)}.
\end{equation}

Henceforward, we identify $C_0^\infty (\Omega )$ to a subset of $C_0^\infty (\mathbb{R}^n)$.

\begin{proposition}\label{proposition-ci}
Let $M>0$ and $s\in \mathbb{R}$. There exist two constants $0<h_s\le 1$ and $C>0$, only depending on $n$, $\Omega$, $\varphi $, $s$ and $M$ so that, for any $V\in C_0^\infty  (\Omega )$ with $\|V\|_{L^n(\Omega )}\le M$, $0<h \le h_s$ and  $u\in C_0^\infty (\Omega)$, we have
\[
\left\|e^{\varphi/h}u\right\|_{H^{s+1}_{scl} (\mathbb{R}^n)}\le Ch\left\|e^{\varphi/h}(-\Delta+V) u\right\|_{H^s_{scl}(\mathbb{R}^n)}.
\]
\end{proposition}

\begin{proof}
Let $\chi\in C_0^\infty (\mathbb{R}^n)$ so that $\chi =1$ near $\overline{\Omega}$. If $u\in C_0^\infty (\Omega )$  then
\[
\|u\|_{H_{scl}^{s+1}(\mathbb{R}^n)}=\|J^1(J^su)\|_{L^2(\mathbb{R}^n)}=\|J^su\|_{H_{scl}^1(\mathbb{R}^n)}.
\]
Hence
\[
\|u\|_{H_{scl}^{s+1}(\mathbb{R}^n)}\le \|\chi J^su\|_{H_{scl}^1(\mathbb{R}^n)}+\|(1-\chi)J^su\|_{H_{scl}^1(\mathbb{R}^n)}.
\]

This, \eqref{ple} and \eqref{ci2} yield
\[
Ch\|u\|_{H_{scl}^{s+1}(\mathbb{R}^n)}\le \sqrt{\epsilon}\|P_{\varphi_\epsilon} (\chi J^su)\|_{L^2(\mathbb{R}^n)}+h^2\|u\|_{H_{scl}^{s+1}(\mathbb{R}^n)}.
\]
Therefore
\begin{equation}\label{u1}
C h\|u\|_{H_{scl}^{s+1}(\mathbb{R}^n)}\le \sqrt{\epsilon}\|P_{\varphi_\epsilon} (\chi J^su)\|_{L^2(\mathbb{R}^n)},
\end{equation}
provided that $h$ is sufficiently small.

Taking into account that $[P_{\varphi_\epsilon} ,\chi]J^su=0$ in $\{\chi =1\}$, one can prove that 
\[
\|[P_{\varphi_\epsilon} ,\chi]J^su\|_{L^2(\mathbb{R}^n)}\le Ch^2\|u\|_{H_{scl}^{s+1}(\mathbb{R}^n)}.
\]

In that case \eqref{u1} yields 
\begin{equation}\label{u3}
Ch\|u\|_{H_{scl}^{s+1}(\mathbb{R}^n)}\le \sqrt{\epsilon}\|P_{\varphi_\epsilon} (J^su)\|_{L^2(\mathbb{R}^n)},
\end{equation}
for sufficiently small $h$.

Now  $[-h^2\Delta,J^s]=0$ gives
\[
[P_{\varphi_\epsilon} ,J^s]u=[2h\nabla \varphi_\epsilon \cdot\nabla ,J^s]u.
\]
We get by applying the commutator estimate \eqref{ce} 
\[
\|[P_{\varphi_\epsilon},J^s]u\|_{L^2(\mathbb{R}^n)}\le Ch\|u\|_{H^s_{scl}(\mathbb{R}^n)}\le C'h\|u\|_{H^{s+1}_{scl}(\mathbb{R}^n)}.
\]
As before, reducing once again $\epsilon_0$ if necessary, this estimate together with \eqref{u3} yield
\[
Ch\|u\|_{H_{scl}^{s+1}(\mathbb{R}^n)}\le \sqrt{\epsilon}\|J^sP_{\varphi_\epsilon} u\|_{L^2(\mathbb{R}^n)}=\sqrt{\epsilon}\|P_{\varphi_\epsilon} u\|_{H^s_{scl}(\mathbb{R}^n)}.
\]
The expected inequality follows then by fixing $\epsilon$. 
\end{proof}

To construct CGO solutions in our case we specify $\varphi$. Precisely, we take as in the preceding section $\varphi (x)=x\cdot \xi$, $\xi \in \mathbb{S}^{n-1}$. Then as a particular case of Proposition \ref{proposition-ci} we have

\begin{proposition}\label{proposition-ci1}
Let $M>0$ and $s\in \mathbb{R}$. There exist two constants $0<h_s\le 1$ and $C>0$, only depending on $n$, $\Omega$, $s$ and $M$ so that, for any $V\in C_0^\infty  (\Omega )$ with $\|V\|_{L^n(\Omega )}\le M$, $\xi \in \mathbb{R}^n$, $0<h \le h_s$ and  $u\in C_0^\infty (\Omega)$, we have
\begin{equation}\label{ci3-1}
\left\|e^{x\cdot \xi/h}u\right\|_{H^{s+1}_{scl} (\mathbb{R}^n)}\le Ch\left\|e^{x\cdot \xi/h}(-\Delta+V) u\right\|_{H^s_{scl}(\mathbb{R}^n)}.
\end{equation}
\end{proposition}

\subsection{CGO solutions}

As in the previous section $P_h$ denotes $P_\varphi$ when $\varphi (x)=x\cdot \xi$, for some $\xi \in \mathbb{S}^{n-1}$.

\begin{proposition}\label{proposition-sol}
Let $M>0$ and $h_0$ as in the preceding proposition. For any $0<h\le h_0$, $V\in L^n (\Omega )$, with $\|V\|_{L^n(\Omega )}\le M$, $\xi \in \mathbb{S}^{n-1}$
 and $f\in L^2(\Omega )$, there exists $w\in H^1(\Omega )$ satisfying $\left[e^{x\cdot \xi/h}(-\Delta +V)e^{-x\cdot \xi/h}\right] w=f$ and
\begin{equation}\label{u4}
\|w\|_{H_{scl}^1(\Omega )}\le Ch\|f\|_{L^2(\Omega )},
\end{equation}
the constant $C$ only depends on $n$, $\Omega$ and $M$.
\end{proposition}

\begin{proof}
We first assume that $V\in C_0^\infty (\Omega )$ with $\|V\|_{L^n(\Omega )}\le M$.

Let $\mathcal{H}=P^\ast_h (C_0^\infty (\Omega))$ that we consider as a subspace of $H_{scl}^{-1}(\mathbb{R}^n)$. Pick $f\in L^2(\Omega)$, extended by $0$ outside $\Omega$,  and define the linear form $\ell$ on $\mathcal{H}$ by
\[
\ell (P^\ast_h v)=(v,h^2f)_{L^2(\mathbb{R}^n)},\quad v \in C_0^\infty (\Omega).
\]

Using that $P^\ast_h(V,\xi)=P_h(\overline{V},-\xi)$, we deduce from \eqref{ci3-1} with $s=-1$, that $\ell$ is well defined and
\[
\left|\ell (P^\ast_h v)\right|\le h^2\|f\|_{L^2(\Omega)}\|v\|_{L^2(\mathbb{R}^n)}\le Ch\|f\|_{L^2(\Omega)}\|P^\ast_h v\|_{H_{scl}^{-1}(\mathbb{R}^n)}.
\]
By Hahn-Banach extension theorem, there exists  a bounded linear form $L$ on $H_{scl}^{-1}(\mathbb{R}^n)$ so that
\[
L (P^\ast_hv)=(v,h^2f)_{L^2(\mathbb{R}^n)},\quad v \in C_0^\infty (\Omega),
\]
and
\begin{equation}\label{u5}
\|L\|_{\left[H_{scl}^{-1}(\mathbb{R}^n)\right]'}\le Ch\|f\|_{L^2(\Omega)}.
\end{equation}
But $\left[H_{scl}^{-1}(\mathbb{R}^n)\right]'$ is identified with $H_{scl}^1(\mathbb{R}^n)$. We then use Riesz's representation theorem to find $z\in H^1(\mathbb{R}^n)$ so that 
\begin{equation}\label{u6}
\|z\|_{H_{scl}^1(\mathbb{R}^n)}=\|L\|_{\left[H_{scl}^{-1}(\mathbb{R}^n)\right]'}
\end{equation}
and
\[
(P^\ast_h v,z) =(v,h^2f)_{L^2(\mathbb{R}^n)},\quad v \in C_0^\infty (\Omega).
\]
In particular, $\left[e^{x\cdot \xi/h}(-\Delta +V)e^{-x\cdot \xi/h}\right] w=f$ in $\Omega$ with $w=z_{|\Omega}$. Furthermore, we see that \eqref{u4} follows readily from \eqref{u5} and \eqref{u6}.

Next, we consider the general case. To this end, we pick $V\in L^n (\Omega )$ with $\|V\|_{L^n(\Omega )}\le M$. Let $(V_k)$ be a sequence in $C_0^\infty (\Omega )$ converging to $V$ in $L^n(\Omega )$. We may then assume that $\|V_k\|_{L^n(\Omega )}\le M+1$ for each $k$. By the previous step, there exists $w_k\in H^1(\Omega )$ satisfying 
\begin{equation}\label{u7}
P_h w_k +h^2(V_k-V)w_k=h^2f
\end{equation}
and 
\begin{equation}\label{u8}
\|w_k\|_{H_{scl}^1(\Omega )}\le Ch\|f\|_{L^2(\Omega )}.
\end{equation}

Inequality \eqref{u8} shows that in particular $(w_k)$ is bounded in $H^1(\Omega )$. Subtracting if necessary a subsequence, we assume that $(w_k)$ converges weakly in $H^1(\Omega)$ to $w\in H^1(\Omega )$ and $(w_k)$ converges strongly to $w$ in $L^2(\Omega )$.

As $P_h w_k$ converges to $P_h w$ in $\mathscr{D}'(\Omega )$ and $(V_k-V)w_k$ converges to $0$ in $L^2(\Omega )$, we get from \eqref{u7} that $P_h w=h^2f$. 

On the other hand, we have in light of \eqref{u8} 
\[
\|w\|_{H^1(\Omega )}\le \liminf_{k}\|w_k\|_{H_{scl}^1(\Omega )}\le Ch\|f\|_{L^2(\Omega )}.
\]
That is \eqref{u4} holds. 
\end{proof}

Let $V\in L^n (\Omega )$ with $\|V\|_{L^n(\Omega )}\le M$ and $\zeta \in \mathbb{S}^{n}$. We seek a solution of $(-\Delta u+V)u=0$ in $\Omega$ of the form
\[
u=e^{-x\cdot (\xi+i \zeta )/h}(1+v),
\]
with $e^{ix\cdot \zeta/h}v$ as in Proposition \ref{proposition-sol}.

We assume that $\xi \in\mathbb{S}^{n-1}$ is so that $\xi \bot \zeta$. Hence
\[
f=-\left[e^{x\cdot \xi/h}(-\Delta +V)e^{-x\cdot \xi/h}\right] (e^{-ix\cdot \zeta/h})=-Ve^{-ix\cdot \zeta/h}.
\]
Then straightforward computations show that $w=e^{ix\cdot \zeta/h}v$ must be a solution of the equation
\[
\left[e^{x\cdot \xi/h}(-\Delta +V)e^{-x\cdot \xi/h}\right]w=f \quad \mbox{in}\; \Omega .
\]

Since
\[
\|Ve^{-ix\cdot \zeta/h}\|_{L^2(\Omega )}\le c_\Omega \|V\|_{L^2(\Omega )}\|e^{-ix\cdot \zeta/h}\|_{L^\infty (\Omega )}\le C,
\]
the constant $C$ only depends on $n$, $\Omega$ and $M$, we get as an immediate consequence of Proposition \ref{proposition-sol} the following result

\begin{theorem}\label{theorem-cgo}
Let $M>0$. Then there exist $0<h_0\le 1$ and $C>0$, only depending on $n$, $\Omega$ and $M$, so that, for any $V\in L^n (\Omega )$ with $\|V\|_{L^n(\Omega )}\le M$, $\xi,\zeta \in \mathbb{S}^{n-1}$ satisfying $\xi \bot \zeta$ and $0<h\le h_0$, the equation
\[
(-\Delta +V)u=0\quad \mbox{in}\; \Omega ,
\]
admits a solution $u\in H^1(\Omega )$ of the form
\[
u=e^{-x\cdot (\xi +i\zeta )/h}(1+v),
\]
where $v\in H^1(\Omega )$ satisfies
\[
\|v\|_{H_{scl}^1(\Omega )}\le Ch.
\]
\end{theorem}

\subsection{Stability inequality}

We define the function $\Psi_\theta $, $\theta >0$, by 
\[
\Psi_\theta (\rho )=|\ln \rho|^{-\theta} +\rho,\quad \rho>0, 
\]
that we extend by $0$ at $\rho=0$.

Hereafter $L_\ast ^n(\Omega )=L^n(\Omega )\cap L_\ast^{n/2}(\Omega )$.

\begin{theorem}\label{theorem-stab}
Let $M>0$ and  $\sigma >0$. Then there exits a constant $C>0$, only depending on $n$, $\Omega$, $M$, $\alpha$ and $\sigma$, so that, for any $V,\tilde{V}\in L_\ast^n(\Omega )$ satisfying $\|V\|_{L^n(\Omega )}\le M$, $\|\tilde{V}\|_{L^n(\Omega )}\le M$, $(V-\tilde{V})\chi_\Omega \in H^\sigma(\mathbb{R}^n)$ and
\[
\|(V-\tilde{V})\chi_\Omega\|_{H^\sigma(\mathbb{R}^n)}\le M, 
\]
we have
\[
C\|V-\tilde{V}\|_{L^2(\Omega )} \le\Psi_\beta (\aleph).
\]
with $\beta =\min (1/2 ,\sigma/n)$ and 
\[
\aleph =\|\Lambda_V-\Lambda_{\tilde{V}}\|_{\mathscr{B}\left(H^{\frac{1}{2}}(\Gamma ),H^{-\frac{1}{2}}(\Gamma )\right)}.
\]
\end{theorem}

\begin{proof}
Pick $V,\tilde{V}\in L_\ast^n(\Omega )$ satisfy $\|V\|_{L^n(\Omega )}\le M$ and $\|\tilde{V}\|_{L^n(\Omega )}\le M$. Let $k,\tilde{k}\in \mathbb{R}^n\setminus\{0\}$ and $\xi \in \mathbb{S}^{n-1}$ so that $k\bot \tilde{k}$, $k\bot \xi$ and $\tilde{k}\bot \xi$. We assume that $|\tilde{k}|=\rho$ with $\rho \ge \rho_0=h_0^{-1}$ where $h_0$ is as Theorem \ref{theorem-cgo}. Let then
\[
h=h(\rho)=\frac{1}{(|k|^2/4+\rho^2)^{1/2}}\; (\le h_0).
\]
Set
\[
\zeta =h(k/2+\tilde{k}),\quad \tilde{\zeta}=h(k/2-\tilde{k})
\]
As we have seen in the proof of Theorem \ref{theorem-bp1}, $\zeta,\tilde{\zeta}\in \mathbb{S}^{n-1}$, $\zeta \bot \xi$, $\tilde{\zeta} \bot \xi$ and $\zeta+\tilde{\zeta}=hk$.

By Theorem \ref{theorem-cgo}, the equation
\[
(-\Delta +V)u=0\quad \mbox{in}\; \Omega 
\]
admits a solution $u\in H^1 (\Omega )$ of the form
\[
u=e^{-x\cdot(\xi +i\zeta )/h}(1+v),
\]
where $v\in H^1(\Omega )$ satisfies
\begin{equation}\label{u9}
\|v\|_{H_{scl}^1(\Omega )}\le Ch.
\end{equation}

 Similarly, the equation
\[
(-\Delta +\tilde{V})u=0\quad \mbox{in}\; \Omega 
\]
admits a solution $\tilde{u}\in H_\Delta (\Omega )$ of the form
\[
\tilde{u}=e^{-x\cdot(-\xi +i\tilde{\zeta} )/h}(1+\tilde{v}),
\]
where $\tilde{v}\in H^1(\Omega )$ satisfies
\begin{equation}\label{u10}
\|\tilde{v}\|_{H_{scl}^1(\Omega )}\le Ch.
\end{equation}

We introduce the temporary notations 
\[
z=(v+\tilde{v}+v\tilde{v})e^{-ix\cdot k},\quad g=\gamma_0 u,\quad \tilde{g}=\gamma_0\overline{\tilde{u}}
\]
and
\[
\aleph =\|\Lambda_V-\Lambda_{\tilde{V}}\|_{\mathscr{B}\left(H^{\frac{1}{2}}(\Gamma ),H^{-\frac{1}{2}}(\Gamma )\right)}.
\]

We find by applying the integral identity \eqref{ii}
\[
\int_\Omega (V-\tilde{V}) e^{-ix\cdot k}dx= -\int_\Omega (V-\tilde{V})zdx+\langle (\Lambda_V-\Lambda_{\tilde{V}})g,\tilde{g}\rangle.
\]
Hence, in light of \eqref{u9} and \eqref{u10}, we deduce
\begin{equation}\label{u11}
|\hat{W}(k)|\le Ch(\rho)+\aleph \|g\|_{H^{1/2}(\Gamma)} \|\tilde{g}\|_{H^{1/2}(\Gamma)},\quad k\in \mathbb{R}^n\setminus\{0\},\; \rho \ge \rho_0,
\end{equation}
with $W=(V-\tilde{V})\chi_\Omega$, where we used that
\[
\left|\int_\Omega (V-\tilde{V})zdx\right|\le \|V-\tilde{V}\|_{L^2(\Omega )}\|z\|_{L^2(\Omega )}\le c_\Omega \|V-\tilde{V}\|_{L^n(\Omega )}\|z\|_{L^2(\Omega )}.
\]
If $c=1+\|x\|_{L^\infty (\Omega )}$, then simple computations show
\begin{align*}
&\|g\|_{H^{1/2}(\Gamma )} \le c_\Omega \|u\|_{H^1(\Omega )}\le Ce^{c/h},
\\
&\|\tilde{g}\|_{H^{1/2}(\Gamma )} \le c_\Omega \|\tilde{u}\|_{H^1(\Omega )}\le Ce^{c/h}.
\end{align*}

These estimates in \eqref{u11} yield
\[
C|\hat{W}(k)|\le h(\rho)+\aleph e^{c/h(\rho )} , \quad k\in \mathbb{R}^n\setminus\{0\},\; \rho \ge \rho_0.
\]
In particular, we have
\[
C|\hat{W}(k)|\le 1/\rho+\aleph e^{c(|k|/2+\rho)} ,\quad \quad k\in \mathbb{R}^n\setminus\{0\},\; \rho \ge \rho_0,
\]
from which we deduce in a straightforward manner, changing if necessary $C$ and $c$,
\begin{equation}\label{u12}
C\int_{|k|\le \rho^{1/n}}|\hat{W}(k)|^2 dk\le 1/\rho+\aleph e^{c\rho} ,\quad  \rho \ge \rho_0.
\end{equation}

On the other hand
\begin{align}
\int_{|k|\ge \rho^{1/n}} |\hat{W}(k )|^2dk &\le \rho^{-2\sigma/n}\int_{|k|\ge h^{-\alpha}} |k^{2\sigma} |\hat{W}(k)|^2dk \label{u13}
\\
&\le \rho^{-2\sigma /n}\|W\|_{H^\sigma (\mathbb{R}^n)}^2. \nonumber
\end{align}

Now inequalities \eqref{u12} and \eqref{u13} together with Planchel-Parseval identity give
\begin{equation}\label{u14}
C\|V-\tilde{V}\|_{L^2(\Omega )} \le 1/\rho^\beta +\aleph e^{c\rho} ,\quad  \rho \ge \rho_0.
\end{equation}
with $\beta =\min \left(1/2 ,\sigma/n\right)$.

Finally, a classical minimization argument applied to \eqref{u14} gives
\[
C\|V-\tilde{V}\|_{L^2(\Omega )} \le\Psi_\beta (\aleph).
\]
The proof is then complete. 
\end{proof}

Let us notice that $\beta=1/2$ in the preceding theorem if $\sigma$ is chosen so that $\sigma \ge n/2$.

The construction of CGO solutions in this section can be extended to the anisotropic case including the magnetic Laplace-Beltrami operator. Precisely in an admissible compact manifold with boundary\footnote{If $(\mathscr{M},g)$ is a compact Riemannian manifold with boundary $\partial \mathscr{M}$, we say that $\mathscr{M}$ is admissible if $\mathscr{M}\Subset \mathbb{R}\times \mathscr{M}_0$, for some (n-1)-dimensional simple manifold $(\mathscr{M}_0,g_0)$ and if $g=c(\mathfrak{e}\oplus g_0)$, where $\mathfrak{e}$ is the Euclidean metric on $\mathbb{R}$ and $c$ is a smooth positive function on $\mathscr{M}$.

A compact Riemannian manifold $(\mathscr{M}_0,g_0)$ with boundary is simple if for any $x\in \mathscr{M}_0$ the exponential map $\exp_x$ with its maximal domain of definition is a diffeomorphism onto $\mathscr{M}_0$, and if $\partial \mathscr{M}_0$ is strictly convex (that is, the second fundamental form of $\partial \mathscr{M}_0\hookrightarrow \mathscr{M}_0$ is positive definite).}. This construction allows the authors in \cite{DKSU} to establish that, in dimension $n\ge 3$, the DN map determines uniquely both the magnetic and the electric potentials (see Theorem 1.7).

\section{Uniqueness for $L^{n/2}$ potentials}

This section is prepared from \cite{DKS} where the reader can find all the details of the results that we state here without proof.

\subsection{Constructing CGO solutions}

The following theorem is the key result that allows the construction of CGO solutions for the Schr\"odinger equations with $L^{n/2}$ potential.

\begin{theorem}\label{theoremG}
 For $|\tau| \ge 4$ outside a countable set, there is a linear operator $G_\tau :L^2(\Omega )\rightarrow H^2(\Omega )$ so that
\begin{align*}
&e^{\tau x_1}(-\Delta)e^{-\tau x_1}G_\tau v=v,\quad v\in L^2(\Omega ),
\\
&G_\tau e^{\tau x_1}(-\Delta)e^{-\tau x_1} v=v,\quad v\in C_0^\infty (\Omega ).
\end{align*}
This operator satisfies
\begin{align*}
&\|G_\tau f\|_{L^2(\Omega )}\le \frac{C}{|\tau|}\|f\|_{L^2(\Omega )},
\\
&\|G_\tau f\|_{H^1(\Omega )}\le C\|f\|_{L^2(\Omega )},
\\
&\|G_\tau f\|_{L^{\overline{n}}(\Omega )}\le C\|f\|_{L^{\underline{n}}(\Omega )},
\end{align*}
the constant $C$ is independent of $\tau$.
\end{theorem}

We first construct CGO solutions for the Schr\"odinger operator without potential. In the rest of this section $\Omega '$ is a fixed open subset of $\mathbb{R}^{n-1}$ so that $\Omega \Subset \mathbb{R}\times \Omega '$.

\begin{lemma}\label{lemma-wp}
Fix $\tilde{x} \in \mathbb{R}^{n-1}\setminus{\overline{\Omega'}}$, $\lambda \in \mathbb{R}$ and let $b\in C^\infty (\mathbb{S}^{n-2})$. If $(r,\theta)$ are the polar coordinates with center $\tilde{x}$, we write $x=(x_1,r,\theta)\in \mathbb{R}^n$. For $|\tau|$ sufficiently large outside a countable set, there exists $u_0\in H^1(\Omega )$ satisfying
\begin{align*}
&-\Delta u_0=0\quad \mbox{in}\; \Omega ,
\\
&u_0=e^{-\tau x_1}\left[e^{-i\tau r}e^{i\lambda (x_1+ir)}b(\theta)+R_0\right],
\end{align*}
where $R_0$ satisfies
\[
|\tau|\|R_0\|_{L^2(\Omega )}+\|R_0\|_{H^1(\Omega )}+\|R_0\|_{L^{\overline{n}}(\Omega )}\le C,
\]
the constant $C$ is independent of $\tau$.
\end{lemma}

\begin{proof}
If $f=-e^{\tau x_1}(-\Delta)e^{-\tau x_1}\left[ e^{-i\tau r}e^{i\lambda (x_1+ir)}b(\theta)\right]$, then we are  reduced to solve the  equation
\[
e^{\tau x_1}(-\Delta)e^{-\tau x_1}R_0=f\quad \mbox{in}\; \Omega .
\]
We have by straightforward computations 
\[
f=\Delta (e^{i\lambda (x_1+ir)}b(\theta)).
\]
Whence
\begin{equation}\label{wp1}
\|f\|_{L^2(\Omega )}+\|f\|_{L^{\underline{n}}(\Omega )}\le C,
\end{equation}
the constant $C$ is independent of $\tau$.

Therefore, according to Theorem \ref{theoremG}, $R_0=G_\tau f$ is a solution of this equation satisfying, in light of \eqref{wp1}, the required  properties. 
\end{proof}

We define the truncation operator $T_k$, $k\ge 1$ is an integer,  on $L^p(\Omega )$, $p\ge 1$, as follows
\[
T_k\varphi (x) = 
\left\{ 
\begin{array}{ll} k&\mbox{if}\; |\varphi (x)|>k , \\ \varphi (x)\quad &\mbox{if}\; |\varphi (x)|\le k, \end{array}
\right.
\quad \varphi \in L^p(\Omega ).
\]

\begin{lemma}\label{truncation-lemma}
Let $\varphi \in L^p(\Omega )$. Then $T_k\varphi\in L^\infty (\Omega )$, $(T_k\varphi)$ converges to $\varphi$ in $L^p(\Omega )$, when $k\rightarrow \infty$,  and $\|T_k\varphi\|_{L^p(\Omega )}\le \|\varphi\|_{L^p(\Omega )}$.
\end{lemma}

\begin{proof}
It is obvious that $|T_k\varphi| \le |\varphi |$ a.e. and $T_k\varphi$ converges  a.e. to $\varphi$. Whence the convergence in $L^p(\Omega )$ holds by virtue of dominated convergence theorem.

The inequality $\|T_k\varphi\|_{L^p(\Omega )}\le \|\varphi\|_{L^p(\Omega )}$ follows readily from  $|T_k\varphi| \le |\varphi |$ a.e..
\end{proof}

\begin{lemma}\label{lemma-multi}
Let $\phi, \psi \in L^n(\Omega )$. Then we have, for large $|\tau|$ outside a countable set,
\begin{equation}\label{multi1}
\|\phi G_\tau \psi \|_{\mathscr{B}(L^2(\Omega ))}\le C\|\phi \|_{L^n(\Omega )}\|\psi\|_{L^n(\Omega )},
\end{equation}
the constant $C$ is independent of $\tau$. Furthermore
\begin{equation}\label{multi2}
\lim_{|\tau| \rightarrow \infty}\| \phi G_\tau \psi\|_{\mathscr{B}(L^2(\Omega ))}=0.
\end{equation}
\end{lemma}

\begin{proof}
In light of properties of $G_\tau$ in Theorem \ref{theoremG}, we get by applying H\"older's inequality, where $f\in L^2(\Omega )$,
\begin{align*}
\|\phi G_\tau \psi f\|_{L^2(\Omega )}&\le C\|\phi\|_{L^n(\Omega)}\|G_\tau \psi f\|_{L^{\overline{n}}(\Omega )}
\\
&\le C\|\phi\|_{L^n(\Omega)}\|\psi f\|_{L^{\underline{n}}(\Omega )}
\\
&\le C\|\phi\|_{L^n(\Omega)}\|\psi\|_{L^n(\Omega)}\|f\|_{L^2(\Omega )},
\end{align*}
the constant $C$ is independent of $\tau$. That is we proved \eqref{multi1}.

Let $\epsilon >0$. Then according to Lemma \ref{truncation-lemma}, we can choose $k$ sufficiently large in such a way that $\phi_0=T_k\phi$ and $\phi_1=\phi -T_k\phi$ are so that $\phi_0\in L^\infty (\Omega )$, $\| \phi_0\|_{L^n (\Omega )}\le \|\phi \|_{L^n(\Omega )}$ and $\|\phi_1\|_{L^n (\Omega)}\le \epsilon /3$.

Similarly we have $\psi=\psi_0+\psi_1$ with $\psi_0\in L^\infty (\Omega )$, $\| \psi_0\|_{L^n (\Omega )}\le \|\psi \|_{L^n(\Omega )}$ and $\|\psi_1\|_{L^n (\Omega)}\le \epsilon /3$.

Using \eqref{multi1}, we find, for some constant $\tilde{C}$ independent of $\tau$,
\begin{align*}
\|\phi G_\tau \psi f\|_{L^2(\Omega )}&\le \|\phi_0 G_\tau \psi_0 f\|_{L^2(\Omega )}+\|\phi_0 G_\tau \psi_1 \|_{L^2(\Omega )} +\|\phi_1 G_\tau \psi \|_{L^2(\Omega )}
\\
&\le \left(\frac{\tilde{C}}{|\tau|}\|\phi_0\|_{L^\infty(\Omega )}\|\phi_1\|_{L^\infty(\Omega )}+\frac{\epsilon}{3}+\frac{\epsilon}{3}\right)\|f\|_{L^2(\Omega )}.
\end{align*}
Therefore $\|\phi G_\tau \psi f\|_{L^2(\Omega )}\le \epsilon\|f\|_{L^2(\Omega )}$, for sufficiently large $|\tau|$. This proves \eqref{multi2}. 
\end{proof}

We are now ready to construct CGO solutions of the Schr\"odinger operator with  $L^{n/2}$ potential.

\begin{theorem}\label{theorem-cgo2}
Let $V\in L^{n/2}(\Omega )$. Fix $\tilde{x} \in \mathbb{R}^{n-1}\setminus{\overline{\Omega'}}$, $\lambda \in \mathbb{R}$ and let $b\in C^\infty (\mathbb{S}^{n-2})$. If $(r,\theta)$ are the polar coordinates with center $\tilde{x}$, we write $x=(x_1,r,\theta)\in \mathbb{R}^n$. For $|\tau|$ sufficiently large outside a countable set, there exists $u\in H^1(\Omega)$ satisfying
\begin{align*}
&(-\Delta +V)u=0\quad \mbox{in}\; \Omega ,
\\
&u=e^{-\tau x_1}\left[e^{-i\tau r}e^{i\lambda (x_1+ir)}b(\theta)+R\right],
\end{align*}
where $R$ satisfies
\[
\|R\|_{L^{\overline{n}}(\Omega )}\le C \quad \mbox{and}\quad \lim_{|\tau|\rightarrow \infty}\|R\|_{L^2(\Omega )}=0,
\]
the constant $C$ is independent of $\tau$.
\end{theorem}

\begin{proof}
We seek $u$ of the form $u=u_0+e^{-\tau x_1}R_1$, where $u_0$ is constructed in Lemma \ref{lemma-wp}. Therefore $R_1$ must be a solution of the equation
\begin{equation}\label{p1}
e^{\tau x_1}(-\Delta +V)e^{-\tau x_1}R_1=-Ve^{\tau x_1}u_0.
\end{equation}

We write 
\[
V(x)=|V(x)|e^{i\vartheta (x)}=|V(x)|^{1/2}W(x)\quad \mbox{with}\quad W(x)=|V(x)|^{1/2}e^{i\vartheta (x)}.
\]
Then we try to find $R_1$ of the form
\[
R_1=G_\tau |V|^{\frac{1}{2}}v.
\]
That is, in light of \eqref{p1}, $v$ should satisfy
\[
\left(1+WG_\tau |V|^{1/2}\right)v=-We^{\tau x_1}u_0.
\]

From Lemma \ref{lemma-multi}, for sufficiently large $|\tau|$, we have $\|WG_\tau |V|^{1/2}\|_{\mathscr{B}(L^2(\Omega ))}\le 1/2$, in which case
\[
v=-\left(1+WG_\tau |V|^{1/2}\right)^{-1}(We^{\tau x_1}u_0).
\]
In the rest of this proof $C$ is a generic constant independent of $\tau$.

We have obviously $\|v\|_{L^2(\Omega )}\le C\|e^{\tau x_1}u_0\|_{L^2(\Omega )}$. This and the estimate in Lemma \ref{lemma-wp} yield $\|v\|_{L^2(\Omega )}\le C$. As
\[
\|R_1\|_{L^{\overline{n}}(\Omega )}\le C\left\||V|^{1/2}v\right\|_{L^{\underline{n}}(\Omega )}\le C\|V\|_{L^{n/2}(\Omega )}\|v\|_{L^2(\Omega)},
\]
we then get, with $R=R_0+R_1$,
\[
\|R\|_{L^{\overline{n}}(\Omega )}\le C.
\]
We already know that $\|R_0\|_{L^2(\Omega )}\le C/|\tau|$. Whence it is enough to prove that $\|R_1\|_{L^2(\Omega )}\rightarrow 0$ when $|\tau| \rightarrow \infty$. 

To this end, as in the preceding proof, for $\epsilon >0$, we decompose $|V|^{1/2}$ in the form
$|V|^{1/2}=\phi +\psi$, with $\phi \in L^\infty (\Omega )$, $\|\phi \|_{L^n(\Omega )}\le \|V\|_{L^{n/2}(\Omega)}^{1/2}$ and $\|\psi\|_{L^n (\Omega )}\le \epsilon$.

In that case, we have
\begin{align*}
\|R_1\|_{L^2(\Omega )}&\le\|G_\tau \phi v\|_{L^2(\Omega )}+C\|\psi v\|_{L^{\underline{n}}(\Omega )}
\\
&\le C\left(\frac{1}{|\tau|}\|\phi\|_{L^\infty(\Omega )}+\|\psi\|_{L^n(\Omega )}\right)
\\
&\le C\left(\frac{1}{|\tau|}\|\phi\|_{L^\infty(\Omega )}+\epsilon\right).
\end{align*}
Hence $\|R_1\|_{L^2(\Omega )}\le C\epsilon$, for $|\tau|$ is sufficiently large.

Let $u_e$ be the solution obtained by the above construction with $\Omega$ substituted by $\Omega_0\Supset \Omega$. Observing that $V$, extended by $0$ outside $\Omega$, belongs to $L^{n/2}(\Omega _0)$ and $u_e\in L^{\overline{n}}(\Omega _0)$ we get, by applying H\"older inequality, that $Vu_e\in L^{\underline{n}}(\Omega_0 )\subset H^{-1}(\Omega_0 )$. 

Pick $\varphi \in C_0^\infty (\Omega_0 )$ satisfying $\varphi=1$ in a neighborhood of $\overline{\Omega}$. Then straightforward computations show that $w=\varphi u_e$ is the unique variational solution in $H_0^1(\Omega_0)$ of the BVP
\[
-\Delta w= f\; \mbox{in}\; \Omega_0 \quad \mbox{and}\quad  w=0\; \mbox{on}\; \partial \Omega _0 ,
\]
with 
\[
f=\varphi Vu_e-2\nabla u_e\cdot \nabla \varphi -\Delta \varphi u_e\in H^{-1}(\Omega_0 ).
\]
We complete then the proof by noting that $u=w_{|\Omega}$ possesses the required properties. 
\end{proof}

\subsection{Uniqueness}

\begin{theorem}\label{theorem-uniqueness}
Let $V,\tilde{V}\in L^{n/2}_\ast(\Omega )$ so that $\Lambda_V=\Lambda_{\tilde{V}}$. Then $V=\tilde{V}$.
\end{theorem}

\begin{proof}
As $\Lambda_V=\Lambda_{\tilde{V}}$, Lemma \ref{integral-identity} gives
\begin{equation}\label{tu1}
\int_\Omega (V-\tilde{V})u\tilde{u}=0,\quad u\in \mathscr{S}_V,\; \tilde{u}\in \mathscr{S}_{\tilde{V}}.
\end{equation}
From Theorem \ref{theorem-cgo2}, for sufficiently large $|\tau |$ outside a countable set, we find $u\in \mathscr{S}_V$ of the form
\begin{equation}\label{tu2}
u=e^{-\tau (x_1+ir)}\left[e^{i\lambda (x_1+ir)}b(\theta)+R\right],
\end{equation}
with $\lambda \in \mathbb{R}$ and $b\in C^\infty(\mathbb{S}^{n-2})$, so that
\begin{equation}\label{tu3}
\|R\|_{L^{\overline{n}}(\Omega )}=O(1)\quad \mbox{and}\quad \|R\|_{L^2(\Omega )}=o(1)\quad \mbox{as}\; |\tau | \rightarrow \infty.
\end{equation}
Here $(r,\theta)$ are the polar coordinates with center $\tilde{x}\in \mathbb{R}^{n-1}\setminus{\overline{\Omega'}}$.

Similarly, for sufficiently large $|\tau |$ outside a countable set, we find $\tilde{u}\in \mathscr{S}_V$ of the form
\begin{equation}\label{tu4}
\tilde{u}=e^{\tau (x_1+ir)}(1+\tilde{R})
\end{equation}
so that
\begin{equation}\label{tu5}
\|\tilde{R}\|_{L^{\overline{n}}(\Omega )}=O(1)\quad \mbox{and}\quad \|\tilde{R}\|_{L^2(\Omega )}=o(1)\quad \mbox{as}\; |\tau | \rightarrow \infty.
\end{equation}

Taking in \eqref{tu1} $u$ and $\tilde{u}$ given respectively by \eqref{tu2} and \eqref{tu4}, we easily get, where $W=V-\tilde{V}$ extended by $0$ outside $\Omega$.
\begin{align}
&\int_{-\infty}^\infty \int_0^\infty\int_{\mathbb{S}^{n-2}} e^{i\lambda (x_1+ir)}W(x_1,r,\theta)b(\theta)dx_1drd\theta \label{tu6}
\\
&=-\int_\Omega |x'-\tilde{x}|^{2-n}W\left[R+e^{i\lambda (x_1+i|x'-\tilde{x}|)}b\left(\frac{x'-\tilde{x}}{|x'-\tilde{x}|}\right)\tilde{R}+R\tilde{R}\right]dx.\nonumber
\end{align}

Let $\epsilon>0$. As we have seen before, we can decompose $W$ in the form $W=W_1+W_2$ with $W_1\in L^\infty (\Omega )$ satisfies $\|W_1\|_{L^{n/2}(\Omega )}\le \|W\|_{L^{n/2}(\Omega )}$ and $W_2\in L^{n/2}(\Omega )$ is so that $\|W_2\|_{L^{n/2}(\Omega )}\le \epsilon$. In that case, we have
\begin{align*}
&C\left| \int_\Omega |x'-\tilde{x}|^{2-n}W\left[R+e^{i\lambda (x_1+i|x'-\tilde{x}|)}b\left(\frac{x'-\tilde{x}}{|x'-\tilde{x}|}\right)\tilde{R}+R\tilde{R}\right]dx\right|
\\
&\hskip 1cm \le \|W_1\|_{L^\infty(\Omega )}\left(\|R\|_{L^2(\Omega )}+\|\tilde{R}\|_{L^2(\Omega )}+\|R\|_{L^2(\Omega )}\|\tilde{R}\|_{L^2(\Omega )}\right)
\\
&\hskip 2cm +\|W_2\|_{L^{n/2}(\Omega)} \left(\|R\|_{L^{\overline{n}}(\Omega )}+\|\tilde{R}\|_{L^{\overline{n}}(\Omega )}+\|R\|_{L^{\overline{n}}(\Omega )}\|\tilde{R}\|_{L^{\overline{n}}(\Omega )}\right).
\end{align*}
This together with \eqref{tu3} and \eqref{tu5} imply that the right hand side of \eqref{tu6} goes to $0$ when $\tau$ tends to $\infty$. That is passing to the limit, when $|\tau|$ tends to $\infty$ in \eqref{tu6}, we find
\[
\int_{-\infty}^\infty \int_0^\infty\int_{\mathbb{S}^{n-2}} e^{i\lambda (x_1+ir)}W(x_1,r,\theta)b(\theta)dx_1drd\theta=0.
\] 

Set 
\[
F(\lambda ,r,\theta )=\int_{-\infty}^\infty e^{i\lambda x_1}W(x_1,r,\theta)dx_1.
\]
Then we obtain by applying Fubini's theorem
\[
\int_0^\infty\int_{\mathbb{S}^{n-2}} e^{-\lambda r}F(\lambda ,r,\theta)b(\theta)drd\theta=0.
\]
As $b$ is arbitrary in $C^\infty (\mathbb{S}^{n-2})$ and $C^\infty (\mathbb{S}^{n-2})$ is dense in $L^2(\mathbb{S}^{n-2})$, we find, by applying once again Fubini's theorem,
\[
\int_0^\infty e^{-\lambda r}F(\lambda ,r,\theta)dr=0,\quad \lambda \in \mathbb{R},\; \theta \in \mathbb{S}^{n-2}.
\]
For $|\lambda|$ small, the injectivity of the attenuated $X$-ray transform yields $F(\lambda ,r,\theta)=0$ for $x'=(r,\theta )\in \Omega'$. Since $F(\cdot ,r,\theta )$ is the Fourier-Laplace transform of $W(\cdot ,r,\theta )$, we deduce that $W=0$ in $\Omega$. That is to say $V=\tilde{V}$ in $\Omega$. 
\end{proof}

This section consists in a simplified version of the result in \cite{DKS}. More precisely, the authors prove in \cite{DKS} the following result:

\begin{theorem}\label{theoremURM}
Let $\Lambda_V$ and $\Lambda_{\tilde{V}}$ defined  by  substituting  de the Laplacian by the Laplace-Beltrami operator on Riemannian manifold $(\mathscr{M},g)$ which is admissible. If $V,\tilde{V}\in L_\ast^{n/2}(\mathscr{M})$ satisfy $\Lambda_V=\Lambda_{\tilde{V}}$ then $V=\tilde{V}$.
\end{theorem}

\section{Borg-Levinson type theorem}

\subsection{$W^{s,p}$ spaces}

Let $\mathcal{O}$ be an open subset of $\mathbb{R}^n$. Let $1<p<\infty$ and $s=m+\sigma$ with $m\in \mathbb{N}$ and $0<\sigma <1$. We denote by $W^{s,p}(\mathcal{O})$ the subspace of functions $u\in W^{m,p}(\mathcal{O})$ so that
\[
[u]_\sigma =\sum_{|\alpha|=m}\int_{\mathcal{O}}\int_{\mathcal{O}}\frac{|\partial^\alpha u(x)-\partial^\alpha u(y)|^p}{|x-y|^{n+p\sigma}}dxdy<\infty .
\]
The space $W^{s,p}(\mathcal{O})$ endowed with its natural norm
\[
\|u\|_{W^{s,p}(\Omega )}=\|u\|_{W^{m,p}(\Omega )}+[u]_\sigma
\]
is a Banach space.

The closure of $C_0^\infty(\mathcal{O})$ in  $W^{s,p}(\mathcal{O})$ is denoted by $W_0^{s,p}(\mathcal{O})$.

When $\Omega$ is of class $C^{k,1}$\footnote{A function is of class $C^{k,1}$ if it is of class $C^k$ and all its partial derivatives of order $k$ are of class $C^{0,1}$}, the construction of $W^{s,p}(\Gamma )$ from $W^{s,p}(\mathbb{R}^{n-1})$, $|s|\le k+1$, is quite similar to that usually used to construct $H^s(\Gamma )$ from $H^s(\mathbb{R}^{n-1})$. That is by means of local cards and a partition of unity. We refer to \cite[Section 1.3.3, page 19]{Gr} for details.

The following trace theorem will be useful in the sequel.

\begin{theorem}\label{trace-theorem}
$($\cite{Gr}$)$ Let $s\in (1,2)$ is so that $s-\frac{1}{\underline{n}}\not\in \mathbb{N}$ and $s-1/\underline{n}=1+\sigma$, $0<\sigma<1$. 
Then the mapping 
\[
u\in C^2(\overline{\Omega})\mapsto (u_{|\Gamma},\partial_\nu u_{\Gamma})\in C^2(\Gamma)\times C^{0,1}(\Gamma )
\]
has unique bounded extension, denoted by $(\gamma_0,\gamma_1)$, as an operator from $W^{s,\underline{n}}(\Omega )$ onto $W^{s-1/\underline{n},\underline{n}}(\Gamma )\times W^{s-1-1/\underline{n},\underline{n}}(\Gamma )$. This operator has a bounded right inverse.
\end{theorem}

\subsection{$W^{2,p}$-regularity}

In the rest of this section, all potentials we consider are assumed to be real valued.

We consider the non homogenous BVP
\begin{equation}\label{nhbvp1}
\left\{
\begin{array}{ll}
(-\Delta +V- \lambda)u=0\quad &\mbox{in}\; \Omega,
\\
u=f &\mbox{on}\; \Gamma .
\end{array}
\right.
\end{equation}

\begin{theorem}\label{theorem-exist}
Pick $V\in L^{n/2}(\Omega )$ and let $\lambda \in \rho (A_V)$. Let $f\in W^{2-1/\underline{n},\underline{n}}(\Gamma )$. Then the BVP \eqref{nhbvp1} has a unique solution $u=u_V(\lambda)(f )\in W^{2,\underline{n}}(\Omega )$. In addition, there exists a constant $C>0$, only depending on $n$, $\Omega$, $V$ and $\lambda$, so that
\begin{equation}\label{nhbvp0}
\|u_V(\lambda)(f)\|_{W^{2,\underline{n}}(\Omega )}\le C\|f\|_{W^{2-1/\underline{n},\underline{n}}(\Gamma )}.
\end{equation}
\end{theorem}

\begin{proof}
Let $\mathcal{E}f\in W^{2,\underline{n}}(\Omega )$ so that $\gamma_0\mathcal{E}f=f$ and
\[
\|\mathcal{E}f\|_{W^{2,\underline{n}}(\Omega )}\le c_\Omega \|f\|_{W^{2-1/\underline{n},\underline{n}}(\Gamma )}.
\]
If  $F=-(-\Delta +V- \lambda)\mathcal{E}f$ then $F\in L^{\underline{n}}(\Omega )$ and
\begin{equation}\label{nhbvp2}
\|F\|_{L^{\underline{n}}(\Omega )}\le C\|f\|_{W^{2-1/\underline{n},\underline{n}}(\Gamma )}.
\end{equation}
By density, we can find a sequence $(F_k)$ in $L^2(\Omega )$ converging in $L^{\underline{n}}(\Omega )$ to $F$. For any $k$, set
\begin{equation}\label{nhbvp3}
v_k=R_V(\lambda )F_k.
\end{equation}
Then $v_k\in H_0^1(\Omega )$ and hence $v_k\in L^{\overline{n}}(\Omega )$. Whence,
\[
-\Delta v_k= -Vv_k+\lambda v_k+F_k \in L^{\underline{n}}(\Omega ).
\]

By \cite[Theorem 9.15, page 241]{GT}, $v_k\in W^{2,\underline{n}}(\Omega )$. Therefore we get in light of \cite[Theorem 9.14, page 240]{GT} that
\[
\|v_k\|_{W^{2,\underline{n}}(\Omega )}\le C_0\|(-\Delta +\lambda_0)v_k\|_{L^{\underline{n}}(\Omega )},
\]
the constants $\lambda_0>0$ and $C_0>0$ only depend on $n$ and $\Omega$. Consequently,
\begin{align*}
\|v_k\|_{W^{2,\underline{n}}(\Omega )}&\le C_0\|-Vv_k+(\lambda+\lambda_0) v_k+F_k\|_{L^{\underline{n}}(\Omega )}
\\
&\le C_0\left(  \|V\|_{L^{n/2}(\Omega )}\|v\|_{L^{\overline{n}}(\Omega )}+|\lambda +\lambda_0|\|v\|_{L^{\underline{n}}(\Omega )}+\|F_k\|_{L^{\underline{n}}(\Omega )}\right).
\end{align*}

From \eqref{nhbvp3}, $(v_k)$ is bounded in $H_0^1(\Omega )$ and hence it is also bounded in $L^{\overline{n}}(\Omega )$. This and the last inequalities show that $(v_k)$ is bounded in $W^{2,\underline{n}}(\Omega )$ with
\begin{equation}\label{nhbvp4}
\|v_k\|_{W^{2,\underline{n}}(\Omega )}\le C\|F_k\|_{L^{\underline{n}}(\Omega )}.
\end{equation}

Now as $W^{2,\underline{n}}(\Omega )$ is reflexive, subtracting if necessary a subsequence, we may assume that $(v_k)$ converges weakly in $W^{2,\underline{n}}(\Omega )\cap W_0^{1,\underline{n}}(\Omega )$ to $v\in W^{2,\underline{n}}(\Omega )\cap W_0^{1,\underline{n}}(\Omega )$. 

Whence $-\Delta v+Vv-\lambda v=F$ in the distributional sense. 

Using that a norm is weakly lower semi-continuous, we get from inequality \eqref{nhbvp4} 
 \begin{equation}\label{nhbvp5}
\|v\|_{W^{2,\underline{n}}(\Omega )}\le C\|F\|_{L^{\underline{n}}(\Omega )}.
\end{equation}
The function $u=\mathcal{E}f+v\in W^{2,\underline{n}}(\Omega )$ is clearly a solution of the BVP \eqref{nhbvp1} and inequality \eqref{nhbvp0}  is a straightforward  consequence of inequalities \eqref{nhbvp2} and \eqref{nhbvp5}. 

The uniqueness of solutions of the BVP \eqref{nhbvp1} follows from the fact that $\lambda$ is not an eigenvalue of $ A_V$. 
\end{proof}

Theorem \ref{theorem-exist} allows us to define a family of DN maps associated $V\in L^{n/2}(\Omega )$: 
\[
\Lambda_V(\lambda ):f\mapsto \gamma_1 u_V(\lambda)(f) ,\quad \lambda\in \rho (A_V).
\]
According to estimate \eqref{nhbvp0} and Theorem \ref{trace-theorem}, $\Lambda_V(\lambda )$ defines a bounded operator between $W^{2-1/\underline{n},\underline{n}}(\Gamma )$ and $W^{1-1/\underline{n},\underline{n}}(\Gamma )$.

\subsection{From spectral data to DN maps}

Note that similar arguments as in the proof of Theorem \ref{theorem-exist} allow us to derive the following estimate, where $V\in L^{n/2}(\Omega )$,
\begin{equation}\label{efe}
\|\phi_V^k\|_{W^{2,\underline{n}}(\Omega )}\le C(|\lambda_k|+1),\quad k\ge 1.
\end{equation}

Henceforward,  we set $\psi_V^k=\gamma_1\phi_V^k$, $k\ge 1$.

\begin{lemma}\label{lemma-derivative}
Let $V\in L^{n/2}(\Omega )$. For any integer $m>n/2+1$, we have 
\begin{equation}\label{series0}
\frac{d^m}{d\lambda ^m}\Lambda_V(\lambda )f=m!\sum_{k\ge 1}\frac{1}{\left(\lambda_V^k-\lambda \right)^{m+1}}(f,\psi_V^k)_{L^2(\Gamma )}\, \psi_V^k.
\end{equation}
\end{lemma}

\begin{proof}
For $f\in W^{2-1/\underline{n},\underline{n}}(\Gamma)$, let $F\in W^{2,\underline{n}}(\Omega )$ satisfying $\Delta F=0$ in $\Omega$, $\gamma_0F=f$ and
\[
\|F\|_{W^{2,\underline{n}}(\Omega )}\le C_\Omega \|f\|_{W^{2-1/\underline{n},\underline{n}}(\Gamma)}.
\]

Using 
\[
u_V(\lambda )(f)=F-R_V(\lambda )((V-\lambda)F),
\]
we obtain
\[
u_V(\lambda )(f)=\sum_{k\ge 1}\frac{1}{\lambda_V^k-\lambda }((\lambda_V^k-V)F,\phi_V^k)_{L^2(\Omega )}\phi_V^k.
\]
It is not hard to check that the above series converges uniformly in $L^2(\Omega )$, with respect to $\lambda$, in each compact set of $\rho (A)$. Consequently, $\lambda \in \rho(A_V) \mapsto  R_V(\lambda )F$ is holomorphic and, for $m\ge 0$,
\begin{equation}\label{series1}
\frac{d^m}{d\lambda ^m}u_V(\lambda )(f)=m!\sum_{k\ge 1}\frac{1}{\left(\lambda_V^k-\lambda \right)^{m+1}}((\lambda_V^k-V)F,\phi_V^k)_{L^2(\Omega )}\phi_V^k.
\end{equation}

Weyl's asymptotic formula \eqref{weyl} together with \eqref{efe} and the inequality
\begin{align*}
 \left|(F,\phi_V^k)_{L^2(\Omega )}\right| &\le \|F\|_{L^{n/2}(\Omega )}\|\phi_V^k\|_{L^{\overline{n}}(\Omega )}
 \\
 &\le C\|f\|_{W^{2-1\underline{n},\underline{n}}(\Gamma)}\|\phi_V^k\|_{W^{2,\underline{n}}(\Omega )}
\end{align*}
yield
\[
\left| \frac{1}{\left(\lambda_V^k-\lambda \right)^{m+1}}((\lambda_V^k-V)F,\phi_V^k)_{L^2(\Omega )}\right|\|\phi_V^k\|_{W^{2,\underline{n}}(\Omega )} \sim \frac{C}{k^{\frac{2(m-1)}{n}}},
\]
as $k\rightarrow \infty$. 

Therefore the series in \eqref{series1} is norm convergent in $W^{2,\underline{n}}(\Omega )$ and hence convergent in $W^{2,\underline{n}}(\Omega )$ (think to the completeness of this Banach space), provided that $m>n/2+1$. 

In consequence, in light of the continuity of the trace operator $\gamma_1: W^{2,\underline{n}}(\Omega )\rightarrow W^{1-1/\underline{n},\underline{n}}(\Omega )$, we get
\[
\frac{d^m}{d\lambda ^m}\Lambda_V(\lambda )f=m!\sum_{k\ge 1}\frac{1}{\left(\lambda_V^k-\lambda \right)^{m+1}}((\lambda_V^k-V)F,\phi_V^k)_{L^2(\Omega )}\psi_V^k.
\]

 But simple calculations based on Green's formula show that
 \[
 ((\lambda_k-V)F,\phi_V^k)_{L^2(\Omega )}=(f,\psi_V^k)_{L^2(\Gamma )}.
 \]
Thus 
\[
\frac{d^m}{d\lambda ^m}\Lambda_V(\lambda )f=m!\sum_{k\ge 1}\frac{1}{\left(\lambda_V^k-\lambda \right)^{m+1}}(f,\psi_V^k)_{L^2(\Gamma )}\psi_V^k.
\]
This is the expect identity. 
\end{proof}

\subsection{Uniqueness}

\begin{proposition}\label{proposition-op}
Let $V,\tilde{V}\in L^{n/2}(\Omega )$ and $0<\epsilon <(n-2)/n$. Then
\begin{equation}\label{opd}
\lim_{\mu \rightarrow \infty}\|\Lambda_V(-\mu^2)-\Lambda_{\tilde{V}}(-\mu ^2)\|_\epsilon =0,
\end{equation}
where $\|\cdot \|_\epsilon$ denotes the natural norm of $\mathscr{B}\left(W^{2-1/\underline{n},\underline{n}}(\Gamma ),W^{1-1/\underline{n}-\epsilon,\underline{n}}(\Gamma )\right)$.
\end{proposition}

The proof of this proposition is based on the following two lemmas. We refer to \cite{Po} for their proof.

\begin{lemma}\label{lemma-re1}
Let $V\in L^{n/2}(\Omega )$. There exist $C>0$ and $\lambda_0>0$, only depending on $n$ and $V$, so that, for any $f\in W^{2-1/\underline{n},\underline{n}}(\Gamma )$ and $\mu \in \mathbb{R}$ with $|\mu| \ge \lambda_0$, we have
\begin{equation}\label{re0}
\|u_V(-\mu ^2,f)\|_{L^{\overline{n}}(\Omega )}\le C\|f\|_{W^{2-1/\underline{n},\underline{n}}(\Gamma )}.
\end{equation}
\end{lemma}

\begin{lemma}\label{lemma-re2}
Let $p\ge \underline{n}$ and $V\in L^{n/2}(\Omega )$. There exist $C>0$ and $\lambda_0>0$, only depending on $n$ and $V$, so that, for any $u\in W^{2,p}(\Omega)\cap W_0^{1,p}(\Omega)$ and $\mu \in \mathbb{R}$ with $|\mu| \ge \lambda_0$, we have
\begin{equation}\label{re1}
\sum_{j=0}^2|\mu |^{2-j}\|u\|_{W^{j,p}(\Omega )}\le C\|(-\Delta +V+\mu ^2)u\|_{L^p(\Omega )}.
\end{equation}
\end{lemma}

\begin{proof}[Proof of Proposition \ref{proposition-op}.]

Pick $f\in W^{2-1/\underline{n},\underline{n}}(\Gamma )$. We have $-\mu ^2\in \rho(A_V)\cap \rho (A_{\tilde{V}})$, for $\mu ^2$ sufficiently large. 

Set $u=u_V(-\mu ^2 ,f)$ and $\tilde{u}=u_{\tilde{V}}(-\mu ^2,f)$. Then $w=u-\tilde{u}$ belongs to $W^{2,p}(\Omega)\cap W_0^{1,p}(\Omega)$ and satisfies
\begin{equation}\label{op0}
(-\Delta +V+\mu ^2)w=(\tilde{V}-V)\tilde{u}\quad \mbox{in}\; \Omega .
\end{equation}

We get by applying Lemma \ref{lemma-re1}
\begin{equation}\label{op1}
\|(\tilde{V}-V)\tilde{u}\|_{L^{\underline{n}}(\Omega)}\le \|V-\tilde{V}\|_{L^{n/2}(\Omega )}\|\tilde{u}\|_{L^{\overline{n}}(\Omega )}\le C\|f\|_{W^{2-1/\underline{n},\underline{n}}(\Gamma )}.
\end{equation}
As $w$ satisfies \eqref{op0}, we obtain in light of \eqref{op1} and \eqref{re1} in Lemma \ref{lemma-re2}
\[
\|w\|_{L^{\underline{n}}(\Omega )}\le C\mu ^{-2}\|f\|_{W^{2-1/\underline{n},\underline{n}}(\Gamma )}\quad \mbox{and}\quad \|w\|_{W^{2,\underline{n}}(\Omega )}\le C\|f\|_{W^{2-1/\underline{n},\underline{n}}(\Gamma )}.
\]
But from usual interpolation inequalities we have
\[
\|w\|_{W^{2-\epsilon,\underline{n}}(\Omega )}\le c_\Omega \|w\|_{L^{\underline{n}}(\Omega )}^{\epsilon/2}\|w\|_{W^{2,\underline{n}}(\Omega )}^{1-\epsilon/2}.
\]

Hence
\[
\|w\|_{W^{2-\epsilon,\underline{n}}(\Omega )}\le C\mu ^{-\epsilon}\|f\|_{W^{2-1/\underline{n},\underline{n}}(\Gamma )}.
\]

Using the continuity of the trace operator $\gamma_1$ (see Theorem \ref{trace-theorem}), we get
\[
\|\gamma_1w\|_{W^{1-\epsilon -1/\underline{n},\underline{n}}(\Gamma )}\le C\mu ^{-\epsilon}\|f\|_{W^{2-1/\underline{n},\underline{n}}(\Gamma )}.
\]
This inequality implies in a straightforward manner \eqref{opd}. 
\end{proof}

We are now ready to prove the following uniqueness result.

\begin{theorem}\label{isp}
Let $V,\tilde{V}\in L^{n/2}(\Omega )$. If 
\[
\lambda_V^k=\lambda_{\tilde{V}}^k \quad \mbox{and}\quad \psi_V^k=\psi_{\tilde{V}}^k,\quad k\ge 1,
\]
then $V=\tilde{V}$.
\end{theorem}

\begin{proof}
In light of \eqref{series0} in Lemma \ref{lemma-derivative}, we get that $\Lambda_V(\lambda )-\Lambda_{\tilde{V}}(\lambda )$ is a polynomial function in $\lambda$. This function is identically equal to zero by Proposition \ref{proposition-op} and hence
\[
\Lambda_V(\lambda )=\Lambda_{\tilde{V}}(\lambda ),\quad \lambda \in \rho(A_V)\cap \rho(A_{\tilde{V}}).
\]
We apply then Theorem \ref{theorem-uniqueness} in order to get $V=\tilde{V}$. 
\end{proof}

Theorem \ref{isp} is borrowed to \cite{Po} where the author considers also the case of partial spectral data. Namely, he proved the following theorem

\begin{theorem}\label{isp2}
Let $V,\tilde{V}\in L^p(\Omega )$ with $p=n/2$ for $n\ge 4$ and $p>n/2$ when $n=3$. If, for an arbitrary positive integer $k_0$,
\[
\lambda_V^k=\lambda_{\tilde{V}}^k \quad \mbox{and}\quad \psi_V^k=\psi_{\tilde{V}}^k,\quad k\ge k_0,
\]
then $V=\tilde{V}$.
\end{theorem}

\begin{proof}[Sketch of the proof]

We pick $V,\tilde{V}\in L^p(\Omega )$ satisfying the assumptions of Theorem \ref{isp2} and we define
\[
D_s=\mathbb{C}\setminus\left(\left\{ \lambda \in \mathbb{C};\; \Re \lambda \ge \frac{s}{2}(\Im \lambda ) ^2-1\right\}\cup \sigma (A_V)\right).
\]

We prove, where $f\in W^{2-1/p,p}(\Gamma )$,
\begin{equation}\label{co1}
\lim_{\lambda \in D_s, |\lambda |\rightarrow \infty}\|\Lambda_V(\lambda)f-\Lambda_{\tilde{V}}(\lambda )f\|_{L^p(\Gamma )}=0.
\end{equation}
 
For $\lambda \in \mathbb{C}\setminus [0,\infty )$ and $\omega \in \mathbb{S}^{n-1}$, we set
\[
\mathfrak{e}_{\omega ,\lambda}(x)=e^{i\sqrt{\lambda}\, x\cdot \omega}.
\]

We define, for $\lambda \in \mathbb{C}\setminus [0,\infty )$ and $\theta,\omega \in \mathbb{S}^{n-1}$,
\begin{align*}
\mathcal{S}_V(\lambda ,\theta ,\omega )&=\int_\Gamma \Lambda_V(\lambda)(\mathfrak{e}_{\lambda ,\omega})\mathfrak{e}_{\lambda ,-\theta}dS(x)
\\
&=\langle \Lambda_V(\lambda)\mathfrak{e}_{\lambda ,\omega},\overline{\mathfrak{e}_{\lambda ,-\theta} }\rangle.
\end{align*}
We define similarly $\mathcal{S}_{\tilde{V}}(\lambda ,\theta ,\omega )$.

One can establish in a straightforward manner the identity
\begin{align}
&\mathcal{S}_V(\lambda ,\theta ,\omega )=\int_\Omega e^{-i\sqrt{\lambda}(\theta -\omega)\cdot x}V(x)dx \label{co2}
\\
&\quad -\frac{\lambda}{2}|\theta -\omega|^2\int_\Omega e^{-i\sqrt{\lambda}(\theta -\omega)\cdot x}dx-\langle R_V(\lambda)(V\mathfrak{e}_{\lambda ,\omega}),V\overline{\mathfrak{e}_{\lambda ,-\theta} }\rangle .\nonumber
\end{align}

Fix $0\ne \xi \in \mathbb{R}^n$ and let $\eta \in \mathbb{S}^{n-1}$ satisfying $\eta \bot \xi$. For any integer $k\ge 1$, define
\[
\theta_k=c_k\eta +\frac{1}{2k}\xi,\quad \omega_k=c_k\eta -\frac{1}{2k}\xi,\quad \sqrt{\lambda_k}= k+i,
\]
where $c_k=\left( 1-|\xi|^2/(4k^2)\right)^{1/2}$.

Then $\theta_k, \omega_k\in \mathbb{S}^{n-1}$,
\begin{align*}
&\sqrt{\lambda_k}(\theta_k-\omega_k)\rightarrow \xi \quad \mbox{as}\; k\rightarrow \infty ,
\\
&\Im \lambda_k \rightarrow \infty \quad \mbox{as}\; k\rightarrow \infty ,
\\
&\sup_k\left| \Im \left( \sqrt{\lambda_k}\, \theta_k\right)\right|,\; \sup_k\left| \Im \left( \sqrt{\lambda_k}\, \omega_k\right)\right| <\infty .
\end{align*}

Noting that  $\sup_k\| \mathfrak{e}_{\lambda_k,\omega_k}\|_{L^\infty (\Omega)}<\infty$ and $\sup_k\| \mathfrak{e}_{\lambda_k,-\theta_k}\|_{L^\infty (\Omega)}<\infty$, \eqref{co1} together with H\"older's inequality yield
\begin{equation}\label{co3}
\lim_{k\rightarrow \infty}\left[ \mathcal{S}_V(\lambda_k ,\theta_k ,\omega_k )-\mathcal{S}_{\tilde{V}}(\lambda_k ,\theta_k ,\omega_k ) \right]=0.
\end{equation}

On the hand, an argument based on Riesz-Thorin's interpolation theorem gives
\[
\lim_{k\rightarrow \infty}\langle R_V(\lambda)(V\mathfrak{e}_{\lambda_k ,\omega_k}),V\overline{\mathfrak{e}_{\lambda_k ,-\theta_k} }\rangle=0.
\]
The same result holds when $V$ is substituted by $\tilde{V}$.

In light of identity \eqref{co2}, we then find
\[
\lim_{k\rightarrow \infty}\left[ \mathcal{S}_V(\lambda_k ,\theta_k ,\omega_k )-\mathcal{S}_{\tilde{V}}(\lambda_k ,\theta_k ,\omega_k ) \right]=\int_\Omega e^{-i\xi \cdot x}(V-\tilde{V})dx,\quad \xi \in \mathbb{R}^n.
\]
Comparing with \eqref{co3}, we end up getting $\mathscr{F}((V-\tilde{V})\chi_\Omega)=0$ and hence $V=\tilde{V}$. 
\end{proof}

\end{document}